\documentclass[12pt,times]{elsarticle}
\usepackage[colorlinks,linkcolor=blue,citecolor=blue]{hyperref}
\usepackage{bookmark}
\usepackage{bm}
\usepackage{amsfonts,amsmath,amssymb,amsthm,cases}
\usepackage{algorithm}      % 提供 algorithm 浮动环境
\usepackage{algpseudocode}  % 提供 algorithmic 排版命令（含 \While、\State 等）
\usepackage{multirow}

\usepackage{longtable}
\usepackage{fullpage}
\usepackage{cleveref}
\usepackage{mathrsfs}
\usepackage{epstopdf}

\usepackage{graphics,graphicx,subfigure}
\usepackage{color,natbib}
\graphicspath{{Figures/}}
\usepackage[section]{placeins}

\newtheorem{theorem}{Theorem}[section]
\newtheorem{lemma}{Lemma}[section]

\newtheorem{remark}{Remark}[section]

\numberwithin{equation}{section}

\renewcommand{\div}{\operatorname{div}}
\newcommand{\curl}{\mathbf{curl}\,}

\newcommand{\T}{\bar{\partial}_t}
\begin{document}
\begin{frontmatter}
	\title{A locking-free virtual element method for the three-dimensional quasi-static electroporoelasticity equations} %\tnoteref{mytitlenote}
	%\tnotetext[mytitlenote]{This work was partially supported by
	%	       National Natural Science Foundation of China under grant number xxxxxxxx.}
	\cortext[mycorrespondingauthor]{Corresponding author}
	\author[sjtu]{Jianguo Huang} %\fnref{hjgfootnote}
	\ead{jghuang@sjtu.edu.cn}
    %\fntext[hjgfootnote]{}
	\address[sjtu]{School of Mathematical Sciences, and MOE-LSC,
		Shanghai Jiao Tong University, Shanghai 200240, China}
	\author[sjtu]{Wenxuan Wang\corref{mycorrespondingauthor}}
	\ead{wangwenxuan@sjtu.edu.cn}
	
	\begin{abstract}
	We develop and analyze a locking-free virtual element method for the three-dimensional quasi-static electroporoelasticity equations. The pressure, electric, magnetic, and displacement fields are discretized by the lowest-order nodal, edge, face, and Stokes-like virtual element spaces, respectively, so that the scheme works on general polyhedral meshes, while the temporal discretization is carried out by a backward Euler scheme. The locking phenomenon is overcome by defining the Stokes-like space on a refined mesh $\widetilde{\mathcal{T}}_h$, on which we can construct an interpolation operator $\mathrm{I}_s$ satisfying $\div(\mathrm{I}_s\bm v)|_K=\Pi^K_0\div\bm v$, which renders the error estimates independent of the Lam\'e constant $\lambda$. Moreover, since the exact sequence $\curl\bm V^e_h=\bm V^f_{h,0}$ holds, the magnetic field is updated explicitly, reducing the four-field system to a three-field one and lowering the computational cost. Numerical experiments are provided to demonstrate the performance of the proposed method.
	\end{abstract}


\end{frontmatter}

\section{Introduction}

Porous materials are widely encountered in engineering systems such as subsurface reservoirs, biomedical implants, and advanced composite structures, where fluid transport and macroscopic response are strongly influenced by the complex topology of internal pores (see \cite{Bear1972,Dullien1991}). When such porous media are subjected to external electromagnetic excitation, the induced electric and magnetic fields interact with the pore-filling fluid and the solid skeleton, giving rise to significant multi-physical coupling effects that directly impact device performance and process efficiency. The seminal work \cite{Pride1994} systematically derived the macroscopic governing equations (i.e., electroporoelasticity equations) for this coupled electrodynamic and acoustic behavior, rigorously linking Maxwell's equations \cite{Monk2003} to Biot's equations \cite{Biot1962} via electrokinetic mechanisms and providing the theoretical foundation for understanding seismoelectric phenomena in fluid-saturated porous media. Subsequently, \cite{PrideHaartsen1996} generalized this model to anisotropic porous media. Given their importance in applications, electroporoelasticity equations have attracted growing interest from researchers (see \cite{HainesPride2006,HuMeir2022,HuMeir2023,LiuZouZhang2025,LiuZouMeir2026}).

Locking phenomena have been extensively studied in linear elasticity and poroelasticity equations \cite{BabuskaSuri1992,OyarzuaRuiz2016}, where various locking-free strategies have been proposed to overcome numerical difficulties caused by near-incompressibility and extreme physical parameters. However, the electroporoelasticity equations involve additional electromagnetic field variables and exhibit much higher complexity than conventional poroelastic models. Hence, developing locking-free numerical methods for electroporoelasticity is of significant theoretical and practical value.

The virtual element method (VEM) is a recent generalization of the finite element method that allows discretizations using arbitrary polygonal or polyhedral meshes, offering exceptional flexibility in handling complex geometries and adaptive meshing while maintaining accuracy and computational efficiency (see \cite{BeiraoBrezzi2013,AhmadAlsaedi2013,BeiraoBrezzi2014}). Theoretical analysis of VEM can be found in \cite{BeiraoLovadina2017,BrennerSung2018,ChenHuang2018}. Over the past decade, VEM has been successfully applied to circumvent locking phenomena in various contexts. For linear elasticity problems, for example, VEM formulations for both compressible and nearly incompressible cases were developed in \cite{BeiraoBrezziMarini2013}. Efficient VEM formulations for compressible and incompressible finite deformation problems were proposed in \cite{WriggersReddy2017}. For poroelasticity problems, locking-free three-field VEM schemes for Biot's consolidation model were introduced in \cite{TangLiu2021}. A mixed displacement-pressure VEM for the Biot's consolidation model was proposed in \cite{WangCai2022} to avoid volumetric locking and ensure stable approximations. VEM has also been extended to electromagnetic problems. For instance, H(div) and H(curl)-conforming VEMs on general polygonal and polyhedral elements were constructed in \cite{BeiraoBrezziMarini2016}. The lowest-order VEM for magnetostatic problems was studied in \cite{BeiraoBrezziDassi2018}. In addition, a VEM discretization for time-dependent Maxwell's equations was developed in \cite{BeiraoDassi2022}.

In this paper, we propose and analyze a locking-free virtual element method for the three-dimensional quasi-static electroporoelasticity equations. The pressure, electric, magnetic, and displacement fields are discretized by the lowest-order nodal, edge, face, and Stokes-like virtual elements, respectively, so that the resulting scheme is well defined on general polyhedral meshes, and a backward Euler scheme is employed for the temporal discretization. In particular, the Stokes-like virtual element space for the displacement field is built on a refined mesh $\widetilde{\mathcal{T}}_h$, which allows us to define an interpolation operator $\mathrm{I}_s$ satisfying $\div(\mathrm{I}_s\bm v)|_K = \Pi^K_0\div\bm v$; it is this discrete divergence property that enables the method to overcome the locking phenomenon. Thanks to the exact sequence property $\curl\bm V^e_h=\bm V^f_{h,0}$ (see Remark \ref{remark-exact-sequence}), the magnetic field can be updated explicitly through \eqref{Hn=Hn-1-En}, so that the four-field system is reduced to a three-field one for $(\bm E^n_h,\bm u^n_h,p^n_h)$ at each time level, which noticeably lowers the computational cost. Moreover, since $\div(\mathrm{I}_f\bm H_0)|_K = \Pi^K_0\div\bm H_0 = 0$ and $\div(\curl\bm E_h)=0$ for all $\bm E_h\in\bm V^e_h$, the fully discrete solution satisfies $\div\bm H^n_h = 0$ for all $1\leq n\leq N$; in other words, Gauss's law for magnetism is inherited exactly by the discrete scheme. Under the assumptions \textbf{A1}--\textbf{A3}, we establish a stability bound and the error estimate $O(\tau+h)$, both with constants independent of the Lam\'e constant $\lambda$, which rigorously justifies the locking-free property of the method. Numerical experiments on CVT, prism, and cube meshes are provided to validate the theoretical results.

Throughout the paper, we use the standard notations for Sobolev spaces and their norms and seminorms (see \cite{Adams1975}). For any open subset $K\subseteq\Omega$, denote by $|K|$ the measure of $K$, and denote by $\mathbb{P}_m(K)$ the set of all polynomials over $K$ with the total degree no more than $m$, with the convention that $\mathbb{P}_{-1}(K)=\{0\}$. For the vector-valued counterparts of $H^m(K)$ and $\mathbb{P}_m(K)$, we respectively write them as $H^m(K,\mathbb{R}^3)$ and $\mathbb{P}_m(K,\mathbb{R}^3)$ while using the same notations for their norms and seminorms as in the scalar case. As usual, $(\cdot,\cdot)_K$ indicates the $L^2$-inner product over $K$. Given any two quantities $a$ and $b$, the symbol ``$a\lesssim b$'' abbreviates ``$a\leq Cb$'' with the hidden positive constant $C$ independent of the mesh size $h$, the time step $\tau$, and the Lam\'e constant $\lambda$. Moreover, ``$a\eqsim b$'' is equivalent to ``$a\lesssim b\lesssim a$''.

\section{The continuous problem}
Let $T>0$ and let $\Omega\subset\mathbb{R}^3$ be a convex polyhedron with Lipschitz boundary $\partial\Omega$. We consider the following quasi-static electroporoelasticity equation on $[0,T]\times\Omega$:
{\allowdisplaybreaks
\begin{align}
	&\varepsilon\partial_t\bm E + \sigma\bm E - \curl\bm H - L\nabla p = \bm j\quad\text{ in } [0,T]\times\Omega, \notag\\
	&\mu\partial_t\bm H + \curl\bm E = \bm 0\quad\text{ in } [0,T]\times\Omega, \notag\\
	&-(\lambda+G)\nabla(\div\bm u) - G\Delta\bm u + \alpha\nabla p = \bm f\quad\text{ in } [0,T]\times\Omega, \notag\\
	&\partial_t(c_0p+\alpha\div\bm u) - \kappa\Delta p + L\div\bm E = g\quad\text{ in } [0,T]\times\Omega, \notag\\
	&\bm E(0,\bm x) = \bm E_0(\bm x),~\bm H(0,\bm x) = \bm H_0(\bm x),~\div\bm H(0,\bm x)=0,~\bm u(0,\bm x) = \bm u_0(\bm x),~p(0,\bm x) = p_0(\bm x)\quad\text{ in } \Omega, \notag\\
	&\bm E(t,\bm x)\times\bm n = \bm 0,~\bm H(t,\bm x)\cdot\bm n = 0,~\bm u(t,\bm x) = \bm 0,~p(t,\bm x) = 0\quad\text{ on } [0,T]\times\partial\Omega.\label{continuous-problem}
\end{align}}
Here, $\bm E$ represents the electric field, $\bm H$ the magnetic field, $\bm u$ the displacement of the solid matrix, and $p$ the pressure within the fluid. The parameter $\varepsilon$ denotes the electric permittivity, $\sigma$ is the electrical conductivity, and $\mu$ represents the magnetic permeability. $L$ is the electrokinetic coupling coefficient. $\alpha \in (0,1]$ is the Biot-Willis coefficient, while $\lambda \in (0,\infty)$ and $G$ stand for the Lam\'e constant and the shear modulus of the porous medium, respectively. Furthermore, $c_0 = 1/M_b$, where $M_b$ is the Biot modulus, and $\kappa = \gamma/\eta$, with $\gamma$ being the permeability and $\eta$ the shear viscosity of the fluid. To simplify the analysis, we assume that $\bm f=\bm 0$ and $\varepsilon$, $\sigma$, $\mu$, $L$, $\alpha$, $\lambda$, $G$, $c_0$, and $\kappa$ are all fixed positive constants. Taking the divergence of the second equation in \eqref{continuous-problem}, we deduce that $\div\bm H(t) = 0~\forall\,t\in[0,T]$. We introduce the function spaces
\begin{align*}
	H(\curl,\Omega) &= \Big\{\bm v\in L^2(\Omega,\mathbb{R}^3):~\curl\bm v\in L^2(\Omega,\mathbb{R}^3)\Big\},	\\
	H_0(\curl,\Omega) &= \Big\{\bm v\in H(\curl,\Omega):~\bm v\times\bm n = \bm 0\text{ on }\partial\Omega\Big\},	\\
	H(\div,\Omega) &= \Big\{\bm v\in L^2(\Omega,\mathbb{R}^3):~\div\bm v\in L^2(\Omega)\Big\},	\\
	H_0(\div,\Omega) &= \Big\{\bm v\in H(\div,\Omega):~\bm v\cdot\bm n = 0\text{ on }\partial\Omega\Big\},
\end{align*}
where $\bm n$ denotes the unit outward normal vector to $\partial\Omega$.

Let $S = H_0(\curl,\Omega)\times H(\curl,\Omega)\times H^1_0(\Omega,\mathbb{R}^3)\times H^1_0(\Omega)$ and $M = L^2(\Omega,\mathbb{R}^3)\times L^2(\Omega,\mathbb{R}^3)\times L^2(\Omega,\mathbb{R}^3)\times L^2(\Omega)$. The variational formulation of \eqref{continuous-problem} is: find $(\bm E,\bm H,\bm u, p)\in L^2(0,T;S)\cap H^1(0,T;M)$ such that for all $(\bm D,\bm B,\bm v,q)\in S$, we have
\begin{align}
	&(\varepsilon\partial_t\bm E(t),\bm D) + (\sigma\bm E(t),\bm D) - (\bm H(t),\curl\bm D) - (L\nabla p(t),\bm D) = (\bm j(t),\bm D), \notag\\
	&(\mu\partial_t\bm H(t),\bm B) + (\curl\bm E(t),\bm B) = 0, \notag\\
	&((\lambda+G)\div\bm u(t),\div\bm v) + (G\nabla\bm u(t),\nabla\bm v) - (p(t),\alpha\div\bm v) = 0, \notag\\
	&(c_0\partial_tp(t),q) + (\alpha\partial_t\div\bm u(t),q) + (\kappa\nabla p(t),\nabla q) - (L\bm E(t),\nabla q) = (g(t),q).\label{variational-formulation}
\end{align}

We assume that
\begin{itemize}\emergencystretch=1em
	\item[\textbf{A1.}] $0<L<\sqrt{\sigma\kappa}$,
	\item[\textbf{A2.}] $\bm j\in H^1(0,T;L^2(\Omega,\mathbb{R}^3))$, $g\in H^1(0,T;L^2(\Omega))$, $\bm E_0\in H_0(\curl,\Omega)\cap H(\div,\Omega)$, $\bm H_0\in H(\curl,\Omega)$, $\bm u_0\in H^1_0(\Omega,\mathbb{R}^3)$, and $p_0\in H^1_0(\Omega)\cap H^2(\Omega)$.
\end{itemize}
If necessary, we further assume that
\begin{itemize}
	\item[\textbf{A3.}] The functions $\bm E_0$, $\bm H_0$, $\bm u_0$, $p_0$, $\bm j$, and $g$ are sufficiently smooth.
\end{itemize}
Under the assumptions \textbf{A1} and \textbf{A2}, the variational formulation \eqref{variational-formulation} admits a unique solution $(\bm E,\bm H,\bm u,p)$ in $L^2(0,T;S)\cap H^1(0,T;M)$, and the following regularity estimates hold (see \cite{LiuZou2026}).
\begin{lemma}\label{lemma-regularity-1}
	Under the assumptions \textbf{A1} and \textbf{A2}, it holds
	\begin{equation*}
		G\|\bm u(t)\|_2 + (\lambda + G)\|\div\bm u(t)\|_1 \lesssim 1\quad \forall\,t\in [0,T],
	\end{equation*}
	where $\bm u$ is the displacement solution of \eqref{variational-formulation}.
\end{lemma}

\begin{lemma}\label{lemma-regularity-2}
	Let $(\bm E,\bm H, \bm u, p)$ be the solution of \eqref{variational-formulation}. Under the assumptions \textbf{A1} and \textbf{A3}, for all $t\in[0,T]$, we have
	\begin{equation*}
		\begin{split}
			&\|\bm E(t)\|^2_2 + \|\bm H(t)\|^2_2 + \|\bm u(t)\|^2_2 + \|p(t)\|^2_2\lesssim 1,	\\
			&\int_{0}^{t}\|\partial_t\bm E(\theta)\|^2_2\,\mathrm{d}\theta + \int_{0}^{t}\|\partial_t\bm H(\theta)\|^2_2\,\mathrm{d}\theta + \int_{0}^{t}\|\partial_t\bm u(\theta)\|^2_2\,\mathrm{d}\theta + \int_{0}^{t}\|\partial_tp(\theta)\|^2_2\,\mathrm{d}\theta \lesssim 1,		\\
			&\int_{0}^{t}\|\partial_{tt}\bm E(\theta)\|^2_0\,\mathrm{d}\theta + \int_{0}^{t}\|\partial_{tt}\bm H(\theta)\|^2_0\,\mathrm{d}\theta + \int_{0}^{t}\|\partial_{tt}\bm u(\theta)\|^2_0\,\mathrm{d}\theta + \int_{0}^{t}\|\partial_{tt}p(\theta)\|^2_0\,\mathrm{d}\theta \lesssim 1,		\\
			&\int_{0}^{t}\|\partial_{tt}\bm u(\theta)\|^2_1\,\mathrm{d}\theta + (\lambda+G)^2\int_{0}^{t}\|\div\partial_t\bm u(\theta)\|^2_1\,\mathrm{d}\theta\lesssim  1.
		\end{split}
	\end{equation*}
\end{lemma}

\section{Virtual element method}
\subsection{Mesh assumptions and the reconstruction of polyhedra}
Let $\{\mathcal{T}_h\}_{h>0}$ be a family of polyhedral decompositions of $\Omega$, where each polyhedron $K\in\mathcal{T}_h$ has diameter $h_K$ and $h:=\max_{K\in\mathcal{T}_h}h_K$. We denote by $\mathcal{F}_h$, $\mathcal{E}_h$ and $\mathcal{V}_h$ the sets of all faces, edges and vertices of $\mathcal{T}_h$, respectively. For each polyhedron $K\in\mathcal{T}_h$, we define
\begin{equation*}
	\mathcal{F}(K) = \{F\in\mathcal{F}_h:~F\subset\partial K\},\quad\mathcal{E}(K) = \{e\in\mathcal{E}_h:~e\subset\partial K\},\quad\mathcal{V}(K) = \{\bm\delta\in\mathcal{V}_h:~\bm\delta\in\partial K\}
\end{equation*}
as its faces, edges and vertices, and the sets $\mathcal{E}(F)$ and $\mathcal{V}(F)$ are defined in a similar way for a face $F\in\mathcal{F}(K)$. Moreover, we write $\bm x_K$ and $\bm x_F$ for the centroids of $K$ and $F$, respectively, and denote by $\bm n_F$ and $\bm t_e$ the unit normal vector to a face $F$ (pointing outward with respect to $K$ whenever $F\in\mathcal{F}(K)$) and a fixed unit tangent vector along an edge $e$, respectively. Throughout this section, we assume that the mesh satisfies the following assumption.
\begin{itemize}
	\item[$\textbf{A4}.$] For every $K\in\mathcal{T}_h$, there exist positive constants $\rho_1$, $\rho_2$ and $\rho_3$, independent of $h_K$, such that
	\begin{itemize}
		\item[$\bullet$] $K$ is star-shaped with respect to a ball of radius $\geq\rho_1h_K$;
		\item[$\bullet$] each face $F\in\mathcal{F}(K)$ is star-shaped with respect to a disk of radius $\geq\rho_2h_K$;
		\item[$\bullet$] the diameter of every face in $\mathcal{F}(K)$ is bigger than $\rho_3h_K$.
	\end{itemize}
\end{itemize}

We next recall the mesh refinement procedure introduced in \cite{HuangWang2026}, which is required in the construction of the Stokes-like virtual element space below. Starting from the original polyhedral mesh $\mathcal{T}_h$, we construct a refined mesh $\widetilde{\mathcal{T}}_h$ by introducing an additional vertex $\bm\delta_F$ at the star-shaped point of each face $F\in\mathcal{F}(K)$ for all $K\in\mathcal{T}_h$ (see Fig.~\ref{refinemesh-example}). The resulting refined polyhedra are denoted by $\widetilde{K}\in\widetilde{\mathcal{T}}_h$, and, for each $\widetilde{K}$, we write $\mathcal{F}(\widetilde{K})$ and $\mathcal{V}(\widetilde{K})$ for the sets of its faces and vertices, respectively. This refinement preserves the mesh assumption $\textbf{A4}$, and $h_{\widetilde{K}} = h_K$ for all $K\in\mathcal{T}_h$.

\begin{figure}[!htb]
	\centering
	\subfigure
	{\includegraphics[width=0.45\textwidth]{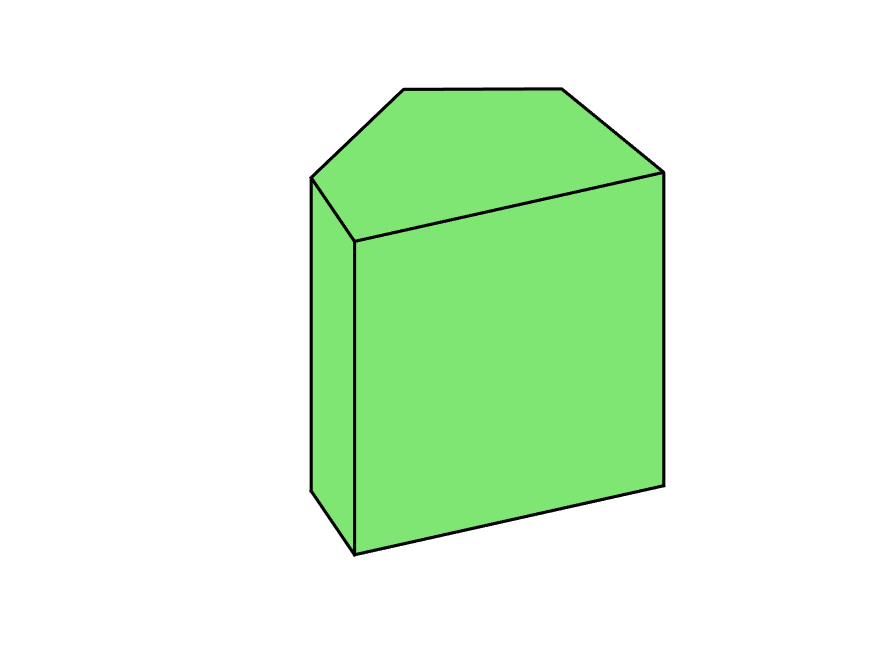}}
	\hspace{0.5cm}
	\subfigure
	{\includegraphics[width=0.45\textwidth]{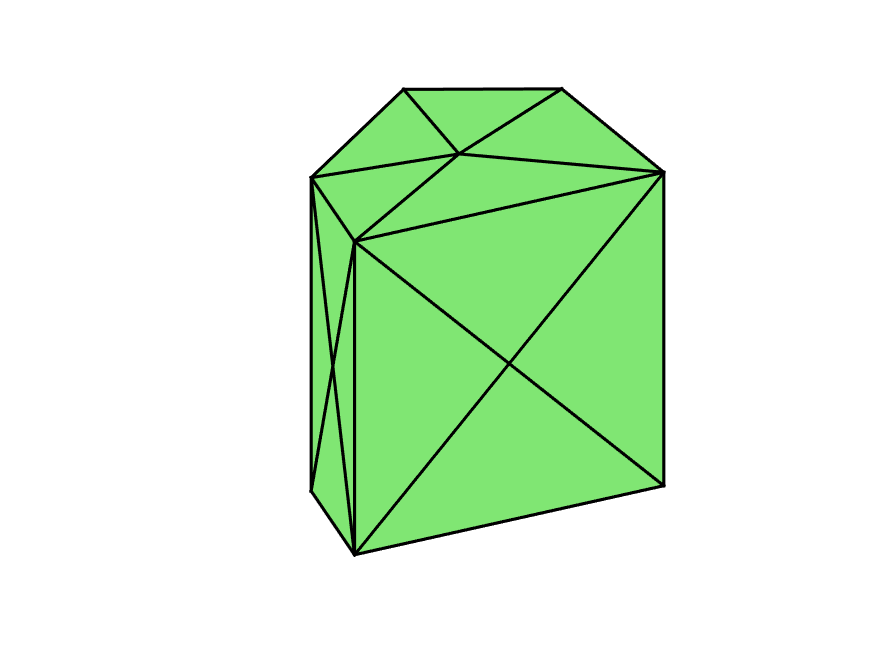}}
	\caption{The initial polyhedron $K$ (left) and the reconstructed polyhedron $\widetilde{K}$ (right).}
	\label{refinemesh-example}
\end{figure}

\subsection{Virtual element spaces}
We construct the virtual element spaces for the four physical fields following the framework of \cite{BeiraoBrezziDassi2018}, in which the lowest-order nodal, edge, and face virtual element spaces are adopted to discretize the pressure, electric, and magnetic fields, respectively. For the displacement field, we define the Stokes-like virtual element space on the refined mesh $\widetilde{\mathcal{T}}_h$ introduced in the previous subsection, following \cite{HuangWang2026}. A key advantage of working on the reconstructed mesh is that we can define an interpolation operator $\mathrm{I}_s$ satisfying $\div(\mathrm{I}_s\bm v)|_K = \Pi^K_0\div\bm v$, which enables the method to overcome the locking phenomenon.

For each $K\in\mathcal{T}_h$, we introduce the following polynomial projection operators:
\begin{itemize}
	\item $\Pi^{\nabla,K}_1:H^1(K)\rightarrow\mathbb{P}_1(K)$
	\begin{equation*}
		\begin{cases}
			\int_K \nabla\Pi^{\nabla,K}_1q\cdot\nabla m\,\mathrm{d}\bm x = \int_K \nabla q\cdot\nabla m\,\mathrm{d}\bm x~\forall m\in\mathbb{P}_1(K),	\\
			\int_{\partial K}\Pi^{\nabla,K}_1q\,\mathrm{d}S = \int_{\partial K}q\,\mathrm{d}S.
		\end{cases}
	\end{equation*}
	\item $\Pi^K_k:L^2(K)\rightarrow\mathbb{P}_k(K)$
	\begin{equation*}
		\int_K \Pi^K_kq\,m\,\mathrm{d}\bm x = \int_K q\,m\,\mathrm{d}\bm x~\forall\,m\in\mathbb{P}_k(K).
	\end{equation*}
	\item $\bm\Pi^{\nabla,K}_1:H^1(K,\mathbb{R}^3)\rightarrow\mathbb{P}_1(K,\mathbb{R}^3)$
	\begin{equation*}
		\begin{cases}
			\int_K \nabla\bm\Pi^{\nabla,K}_1\bm v:\nabla \bm m\,\mathrm{d}\bm x = \int_K \nabla\bm v:\nabla \bm m\,\mathrm{d}\bm x~\forall \bm m\in\mathbb{P}_1(K,\mathbb{R}^3),	\\
			\int_{\partial K}\bm\Pi^{\nabla,K}_1\bm v\,\mathrm{d}S = \int_{\partial K}\bm v\,\mathrm{d}S.
		\end{cases}
	\end{equation*}
	\item $\bm\Pi^K_0:L^2(K,\mathbb{R}^3)\rightarrow\mathbb{P}_0(K,\mathbb{R}^3)$
	\begin{equation*}
		\int_K \bm\Pi^K_0\bm E\,\mathrm{d}\bm x = \int_K \bm E\,\mathrm{d}\bm x.
	\end{equation*}
\end{itemize}

We first define the nodal virtual element space. On each $K\in\mathcal{T}_h$, let
\begin{equation*}
	U^r_h(K) = \Big\{q\in H^1(K):~\Delta q = 0\text{ in } K,~q|_{\partial K}\in C^0(\partial K),~q|_F\in V^r_h(F)~\forall\,F\in\mathcal{F}(K)\Big\},
\end{equation*}
where for each face $F\in\mathcal{F}(K)$,
\begin{align*}
V^r_h(F) = \Big\{q\in H^1(F)\cap C^0(F):~\Delta_Fq\in\mathbb{P}_0(F),~q|_e\in\mathbb{P}_1(e)~\forall\,e\in\mathcal{E}(F),\\
\int_F\nabla q\cdot(\bm x-\bm x_F)\,\mathrm{d}S=0\Big\}.
\end{align*}
The corresponding enhanced nodal virtual element space is defined by
\begin{align*}
	V^r_h(K) = \Big\{q\in H^1(K):~\Delta q \in\mathbb{P}_1(K)\text{ in } K,~q|_{\partial K}\in C^0(\partial K),~q|_F\in V^r_h(F)~\forall\,F\in\mathcal{F}(K),\\
	(q-\Pi^{\nabla,K}_1q,m)_K = 0~\forall\,m\in\mathbb{P}_1(K)\Big\}.
\end{align*}
According to the definition of $V^r_h(K)$,
\begin{equation*}
	\Pi^{\nabla,K}_1q = \Pi^K_1q\quad\forall\,q\in V^r_h(K).
\end{equation*}
In both $V^r_h(K)$ and $U^r_h(K)$, the DoFs consist of the pointwise values $q(\bm\delta)$ on each vertex $\bm\delta\in\mathcal{V}(K)$.
Gluing these local spaces together with a global continuity condition yields the global nodal virtual element spaces
\begin{align*}
	&U^r_h = \Big\{q\in H^1_0(\Omega):~ q|_K\in U^r_h(K)~\forall K\in\mathcal{T}_h\Big\},\\
	&V^r_h = \Big\{q\in H^1_0(\Omega):~ q|_K\in V^r_h(K)~\forall K\in\mathcal{T}_h\Big\}.
\end{align*}

For the approximation of the electric field, we employ the edge virtual element space:
\begin{align*}
	\bm V^e_h(K) =\Big\{&\bm E\in L^2(K,\mathbb{R}^3):~\bm E^{\bm\tau}|_F\in \bm V^e_h(F)\quad\forall\,F\in\mathcal{F}(K),\quad \div\bm E = 0 \text{ in } K,	\\
	&\curl(\curl\,\bm E)\in\mathbb{P}_0(K,\mathbb{R}^3),~\bm E\cdot\bm t_e \text{ continuous at each edge } e\in\mathcal{E}(K), \\
	&\text{ and }\int_K(\curl\bm E)\cdot \big((\bm x-\bm x_K)\times\bm p\big)\,\mathrm{d}\bm x = 0\quad\forall\,\bm p\in\mathbb{P}_0(K,\mathbb{R}^3)\Big\},
\end{align*}
where $\bm E^{\bm\tau}$ denotes the tangential component of $\bm E$, the two-dimensional edge space on a face $F\in\mathcal{F}(K)$ is given by
\begin{align*}
	\bm V^e_h(F) = \Big\{\bm E\in L^2(F,\mathbb{R}^2):~\div_F\bm E\in\mathbb{P}_0(F),~\mathrm{rot}_F\bm E\in\mathbb{P}_0(F), \\
	\bm E\cdot\bm t_e\in\mathbb{P}_0(e)~\forall\,e\in\mathcal{E}(F),~\int_F\bm E\cdot(\bm x-\bm x_F)\,\mathrm{d}S = 0\Big\},
\end{align*}
In $\bm V^e_h(K)$, the DoFs consist of the tangential components $\bm E\cdot\bm t_e$ on each edge $e\in\mathcal{E}(K)$.
The global edge virtual element space is defined as
\begin{equation*}
	\bm V^e_h = \Big\{\bm E\in H_0(\curl,\Omega):~\bm E|_K\in\bm V^e_h(K)~\forall\,K\in\mathcal{T}_h\Big\}.
\end{equation*}

For the magnetic field, we use the face virtual element space:
\begin{align*}
	\bm V^f_h(K) = \Big\{\bm H\in L^2(K,\mathbb{R}^3):~\bm H\cdot\bm n_F\in\mathbb{P}_0(F)\quad\forall\,F\in\mathcal{F}(K),\quad \div\bm H\in\mathbb{P}_0(K),\\
	\curl\,\bm H\in\mathbb{P}_0(K,\mathbb{R}^3), \text{ and } \int_K\bm H\cdot\big((\bm x-\bm x_K)\times\bm p\big)\,\mathrm{d}\bm x = 0\quad\forall\,\bm p\in\mathbb{P}_0(K,\mathbb{R}^3)\Big\}.
\end{align*}
In $\bm V^f_h(K)$, the DoFs consist of the normal components $\bm H\cdot\bm n_F$ on each face $F\in\mathcal{F}(K)$.
The global face virtual element space is given by
\begin{equation*}
	\bm V^f_h = \Big\{\bm H\in H_0(\div,\Omega):~\bm H|_K\in\bm V^f_h(K)~\forall\,K\in\mathcal{T}_h\Big\}.
\end{equation*}

\begin{remark}\label{remark-exact-sequence}
	Let $\bm V^e_{h,0} = \{\bm E\in\bm V^e_h:\curl\bm E=0\}$ and $\bm V^f_{h,0} = \{\bm H\in\bm V^f_h:\div\bm H=0\}$. As observed in \cite{BeiraoBrezziDassi2018}, the following two identities hold:
	\begin{equation*}
		\bm V^e_{h,0} = \nabla U^r_h,\quad \bm V^f_{h,0} = \curl\bm V^e_h.
	\end{equation*}
	In particular, the spaces $U^r_h$, $\bm V^e_h$, and $\bm V^f_h$ form an exact sequence.
\end{remark}

The Stokes-like virtual element space is defined as
\begin{equation*}
	\bm V^s_h(K)=\Big\{\bm v \in H^1( K ,\mathbb{R}^3):~\bm v|_{\partial K}\in\mathbb{B}(\partial K,\mathbb{R}^3),\,\begin{cases}
		-\Delta \bm v-\nabla s = 0\\
		\div\bm v\in\mathbb{P}_0(K),
	\end{cases}\textrm{for some } s\in L^2(K)\Big\},
\end{equation*}
where $\mathbb{B}(\partial K,\mathbb{R}^3)$ is the function space on the boundary $\partial K$ defined by
\begin{equation*}
	\mathbb{B}(\partial K,\mathbb{R}^3)=\left\{\bm v\in C^0(\partial K,\mathbb{R}^3):~\bm v|_{F}\in\mathbb{P}_1(F,\mathbb{R}^3) \quad\forall\,F\in\mathcal{F}(\widetilde{K})\right\}.
\end{equation*}
In $\bm V^s_h(K)$, the DoFs consist of the vertex values $\bm v(\bm\delta)$ on each vertex $\bm\delta\in\mathcal{V}(\widetilde{K})$. The global Stokes-like virtual element space is given by
\begin{equation*}
	\bm V^s_h = \Big\{\bm v\in H^1_0(\Omega,\mathbb{R}^3):\bm v|_K\in\bm V^s_h(K)~\forall\,K\in\mathcal{T}_h\Big\}.
\end{equation*}

For the spaces introduced above, we define the corresponding interpolation operators. For every $K\in\mathcal{T}_h$,
\begin{itemize}
	\item $\mathrm{I}_r:H^2(\Omega)\cap H^1_0(\Omega)\rightarrow V^r_h$ is defined by
	\begin{equation*}
		\mathrm{I}_rq(\bm\delta) = q(\bm\delta)\quad \forall\,\bm\delta\in\mathcal{V}(K).
	\end{equation*}
	
	\item $\mathrm{I}_s:H^2(\Omega,\mathbb{R}^3)\cap H^1_0(\Omega,\mathbb{R}^3)\rightarrow\bm V^s_h$ is defined by
	\begin{align*}
		\mathrm{I}_s\bm v(\bm\delta) &= \bm v(\bm\delta)\quad\forall\,\bm\delta\in\mathcal{V}(K),\\
		\int_F\mathrm{I}_s\bm v\,\mathrm{d}S &= \int_F\bm v\,\mathrm{d}S\quad\forall\,F\in\mathcal{F}(K).
	\end{align*}
	
	\item $\mathrm{I}_e:H^2(\Omega,\mathbb{R}^3)\cap H_0(\curl,\Omega)\rightarrow\bm V^e_h$ is defined by
	\begin{equation*}
		\int_e \mathrm{I}_e\bm E\cdot\bm t_e\,\mathrm{d}s = \int_e \bm E\cdot\bm t_e\,\mathrm{d}s\quad\forall\,e\in\mathcal{E}(K).
	\end{equation*}
	
	\item $\mathrm{I}_f:H^1(\Omega,\mathbb{R}^3)\cap H_0(\div,\Omega)\rightarrow\bm V^f_h$ is defined by
	\begin{equation*}
		\int_F\mathrm{I}_f\bm H\cdot\bm n_F\,\mathrm{d}S = \int_F\bm H\cdot\bm n_F\,\mathrm{d}S\quad\forall\,F\in\mathcal{F}(K).
	\end{equation*}
\end{itemize}
These interpolation operators satisfy the following optimal error estimates (see \cite{BeiraoDassi2022,HuangYu2023,HuangWang2026}).

\begin{lemma}\label{lemma-interpolation}
	Let $K\in\mathcal{T}_h$. For all $q\in H^2(\Omega)\cap H^1_0(\Omega)$, $\bm v\in H^2(\Omega,\mathbb{R}^3)\cap H^1_0(\Omega,\mathbb{R}^3)$, $\bm E\in H^2(\Omega,\mathbb{R}^3)\cap H_0(\curl,\Omega)$, and $\bm H\in H^1(\Omega,\mathbb{R}^3)\cap H_0(\div,\Omega)$, the following estimates hold:
	\begin{equation*}
		\begin{split}
			&\|q-\mathrm{I}_rq\|_{0,K} + h_K|q-\mathrm{I}_rq|_{1,K}\lesssim h^2_K|q|_{2,K},	\\
			&\|\bm v-\mathrm{I}_s\bm v\|_{0,K} + h_K|\bm v-\mathrm{I}_s\bm v|_{1,K}\lesssim h^2_K|\bm v|_{2,K}\text{ with }\div(\mathrm{I}_s\bm v)|_K = \Pi^K_0\div\bm v,	\\
			&\|\bm E-\mathrm{I}_e\bm E\|_{0,K} + \|\curl(\bm E-\mathrm{I}_e\bm E)\|_{0,K} \lesssim h_K\|\bm E\|_{2,K},	\\
			&\|\bm H-\mathrm{I}_f\bm H\|_{0,K}\lesssim h_K|\bm H|_{1,K}\text{ with }\div(\mathrm{I}_f\bm H)|_K = \Pi^K_0\div\bm H.
		\end{split}
	\end{equation*}
\end{lemma}

In addition, we introduce the following projection operators:
\begin{itemize}
	\item $\bm P_e:H_0(\curl,\Omega)\rightarrow\bm V^e_h$ is defined by
	\begin{equation*}
		\begin{cases}
			[\curl(\bm P_e\bm E),\curl\bm D_h]_f = (\curl\bm E,\curl\bm D_h)\quad \forall\,\bm D_h\in\bm V^e_h,	\\
			[\bm P_e\bm E,\nabla q_h]_e = (\bm E,\nabla q_h)\quad \forall\,q_h\in U^r_h.
		\end{cases}
	\end{equation*}
	
	\item $\bm P_f:L^2(\Omega,\mathbb{R}^3)\rightarrow \curl(\bm V^e_h)$ is defined by
	\begin{equation*}
		[\bm P_f\bm H,\bm B_h]_f = (\bm H,\bm B_h)\quad \forall\,\bm B_h\in\curl(\bm V^e_h).
	\end{equation*}
\end{itemize}
The following lemma shows their approximation properties (see \cite{BeiraoDassi2022}).

\begin{lemma}\label{lemma-projection}
	For all $\bm E\in H^2(\Omega,\mathbb{R}^3)\cap H_0(\curl,\Omega)$ and divergence-free $\bm H\in H^1(\Omega,\mathbb{R}^3)$, we have
	\begin{equation*}
		\begin{split}
			&\|\bm E-\bm P_e\bm E\|_0 + \|\curl(\bm E-\bm P_e\bm E)\|_0\lesssim h\|\bm E\|_2,	\\
			&\|\bm H - \bm P_f\bm H\|_0\lesssim h\|\bm H\|_1.
		\end{split}
	\end{equation*}
\end{lemma}

\subsection{The discrete problem}
To construct a suitable discrete formulation, we equip each virtual element space with the following bilinear forms. Let $K\in\mathcal{T}_h$. For all $p_h, q_h\in V^r_h(K)$, $\bm u_h,\bm v_h\in\bm V^s_h(K)$, $\bm E_h, \bm D_h\in\bm V^e_h(K)$, and $\bm H_h, \bm B_h\in\bm V^f_h(K)$, we define
{\allowdisplaybreaks
\begin{align*}
	&[p_h,q_h]_{r,K} = (\nabla\Pi^{\nabla,K}_1p_h,\nabla\Pi^{\nabla,K}_1q_h)_K + h_KS^K_1((I-\Pi^{\nabla,K}_1)p_h,(I-\Pi^{\nabla,K}_1)q_h),	\\
	&[p_h,q_h]_{m,K} = (\Pi^K_1p_h,\Pi^K_1q_h)_K + h^3_KS^K_2((I-\Pi^K_1)p_h,(I-\Pi^K_1)q_h),\\
	&[\bm u_h,\bm v_h]_{s,K} = (\nabla\bm\Pi^{\nabla,K}_1\bm u_h,\nabla\bm\Pi^{\nabla,K}_1\bm v_h)_K + h_KS^K_3((I-\bm\Pi^{\nabla,K}_1)\bm u_h,(I-\bm\Pi^{\nabla,K}_1)\bm v_h),\\
	&[\bm E_h,\bm D_h]_{e,K} = (\bm\Pi^K_0\bm E_h,\bm\Pi^K_0\bm D_h)_K + h^2_KS^K_4((I-\bm\Pi^K_0)\bm E_h,(I-\bm\Pi^K_0)\bm D_h),\\
	&[\bm H_h,\bm B_h]_{f,K} = (\bm\Pi^K_0\bm H_h,\bm\Pi^K_0\bm B_h)_K + h_KS^K_5((I-\bm\Pi^K_0)\bm H_h,(I-\bm\Pi^K_0)\bm B_h).
\end{align*}
}
Here
\begin{equation*}
	\begin{split}
		&S^K_1(p_h,q_h) = S^K_2(p_h,q_h) = \sum_{\bm\delta\in\mathcal{V}(K)}p_h(\bm\delta)q_h(\bm\delta),\quad S^K_3(\bm u_h,\bm v_h) = \sum_{\bm\delta\in\mathcal{V}(\widetilde{K})}\bm u_h(\bm\delta)\cdot\bm v_h(\bm\delta),	\\
		&S^K_4(\bm E_h,\bm D_h) = \sum_{e\in\mathcal{E}(K)}(\bm E_h\cdot\bm t_e,\bm D_h\cdot\bm t_e)_e,\quad S^K_5(\bm H_h,\bm B_h) = \sum_{F\in\mathcal{F}(K)}(\bm H_h\cdot\bm n_F,\bm B_h\cdot\bm n_F)_F.
	\end{split}
\end{equation*}
The corresponding norms are given by
\begin{equation*}
	\begin{split}
		&\||p_h\||_{m,K} = \sqrt{[p_h,p_h]_{m,K}},~\||\bm u_h\||_{s,K} = \sqrt{[\bm u_h,\bm u_h]_{s,K}},\\
		&\||\bm E_h\||_{e,K} = \sqrt{[\bm E_h,\bm E_h]_{e,K}},~\||\bm H_h\||_{f,K} = \sqrt{[\bm H_h,\bm H_h]_{f,K}}.
	\end{split}
\end{equation*}
From the norm equivalence of virtual elements (see \cite{ChenHuang2018,BeiraoMascotto2023,HuangWang2026}), we have
\begin{equation*}
	\begin{split}
		&\||p_h\||_{m,K}\eqsim\|p_h\|_{0,K},~\||\bm u_h\||_{s,K}\eqsim|\bm u_h|_{1,K},\\
		&\||\bm E_h\||_{e,K}\eqsim\|\bm E_h\|_{0,K},~\||\bm H_h\||_{f,K}\eqsim\|\bm H_h\|_{0,K}.
	\end{split}
\end{equation*}

With the above spaces and bilinear forms in mind, we are now ready to present the fully discrete virtual element approximation of the problem \eqref{continuous-problem}. For a given time step $\tau>0$ and a sequence of discrete times $t_n=n\tau$, we use the notation
\begin{equation*}
	\bm E_n = \bm E(t_n),~\bm H_n = \bm H(t_n),~\bm u_n = \bm u(t_n),~p_n = p(t_n),~\bm j_n = \bm j(t_n),~g_n = g(t_n).
\end{equation*}
Let $S_h = \bm V^e_h\times\bm V^f_h\times\bm V^s_h\times V^r_h$. The discrete variational formulation is: find $(\bm E^n_h,\bm H^n_h,\bm u^n_h,p^n_h)\in S_h$ such that
\begin{align}
	&[\varepsilon\T\bm E^n_h,\bm D_h]_e + [\sigma\bm E^n_h,\bm D_h]_e - [\bm H^n_h,\curl\bm D_h]_f - L\sum_{K\in\mathcal{T}_h}[\nabla \Pi^K_1p^n_h,\bm D_h]_{e,K} = \sum_{K\in\mathcal{T}_h}[\bm\Pi^K_0\bm j_n,\bm D_h]_{e,K},\notag\\
	&[\mu\T\bm H^n_h,\bm B_h]_f + [\curl\bm E^n_h,\bm B_h]_f = 0,\notag\\
	&((\lambda+G)\div\bm u^n_h,\div\bm v_h) + [G\bm u^n_h,\bm v_h]_s - (p^n_h,\alpha\div\bm v_h) = 0,\notag\\
	&[c_0\T p^n_h,q_h]_m + (\alpha\T\div\bm u^n_h,q_h) + [\kappa p^n_h,q_h]_r - L\sum_{K\in\mathcal{T}_h}[\bm E^n_h,\nabla\Pi^K_1 q_h]_{e,K} = \sum_{K\in\mathcal{T}_h}(\Pi^K_1g_n,q_h)_K.\label{discrete-problem}
\end{align}
Here $(\bm D_h,\bm B_h,\bm v_h,q_h)\in S_h$, $(\bm E^0_h,\bm H^0_h,\bm u^0_h,p^0_h) = (\mathrm{I}_e\bm E_0,\mathrm{I}_f\bm H_0,\mathrm{I}_s\bm u_0,\mathrm{I}_rp_0)$ and $\T(\cdot)^n = ((\cdot)^n-(\cdot)^{n-1})/\tau$.

By definition, the second equation in \eqref{discrete-problem} can be rewritten as
\begin{equation*}
	[\bm H^n_h-\bm H^{n-1}_h + \tau\mu^{-1}\curl\bm E^n_h,\bm B_h]_f = 0\quad \forall\,\bm B_h\in\bm V^f_h.
\end{equation*}
According to Remark \ref{remark-exact-sequence},
\begin{equation}\label{Hn=Hn-1-En}
	\bm H^n_h = \bm H^{n-1}_h - \tau\mu^{-1}\curl\bm E^n_h.
\end{equation}
Recalling Lemma \ref{lemma-interpolation}, we obtain
\begin{equation*}
	(\div\bm H^0_h)|_K = \Pi^K_0\div\bm H_0=0\quad\forall\,K\in\mathcal{T}_h.
\end{equation*}
The above equation and \eqref{Hn=Hn-1-En} imply that
\begin{equation*}
	\div\bm H^n_h = 0.
\end{equation*}
We substitute \eqref{Hn=Hn-1-En} in the first equation of \eqref{discrete-problem} and get
\begin{equation*}
	\begin{split}
		&[(\varepsilon+\tau\sigma)\bm E^n_h,\bm D_h]_e + [\tau^2\mu^{-1}\curl\bm E^n_h,\curl\bm D_h]_f - \tau L\sum_{K\in\mathcal{T}_h}[\nabla \Pi^K_1p^n_h,\bm D_h]_{e,K}	\\
		=& \sum_{K\in\mathcal{T}_h}[\tau\bm\Pi^K_0\bm j_n+\varepsilon\bm E^{n-1}_h,\bm D_h]_{e,K} + [\tau\bm H^{n-1}_h,\curl\bm D_h]_f.
	\end{split}
\end{equation*}
This reformulation allows us to reduce the computational effort in solving \eqref{discrete-problem}.

\section{Stability analysis}

\begin{theorem}
	Let $(\bm E^n_h,\bm H^n_h,\bm u^n_h,p^n_h)$ be the solution of \eqref{discrete-problem}. For all $0\leq n\leq N$, we have
	\begin{equation*}
		\varepsilon\||\bm E^n_h\||^2_e + \mu\||\bm H^n_h\||^2_f + c_0\||p^n_h\||^2_m + (\lambda+G)\|\div\bm u^n_h\|^2_0 + G\||\bm u^n_h\||^2_s\lesssim 1.
	\end{equation*}
\end{theorem}

\begin{proof}
	Taking $(\bm D_h,\bm B_h,\bm v_h,q_h) = (\bm E^n_h,\bm H^n_h,\T\bm u^n_h,p^n_h)$ in \eqref{discrete-problem}, we obtain
	\begin{align}
		&[\varepsilon\T\bm E^n_h,\bm E^n_h]_e + [\mu\T\bm H^n_h,\bm H^n_h]_f + [c_0\T p^n_h,p^n_h]_m + ((\lambda+G)\div\bm u^n_h,\div\T\bm u^n_h) \notag\\
		&+ [G\T\bm u^n_h,\bm u^n_h]_s + [\sigma\bm E^n_h,\bm E^n_h]_e + [\kappa p^n_h,p^n_h]_r - 2L\sum_{K\in\mathcal{T}_h}[\nabla\Pi^K_1 p^n_h,\bm E^n_h]_{e,K} \notag\\
		&= \sum_{K\in\mathcal{T}_h}[\bm\Pi^K_0\bm j_n,\bm E^n_h]_{e,K} + \sum_{K\in\mathcal{T}_h}(\Pi^K_1g_n,p^n_h)_K.\label{stability-1}
	\end{align}
	
	Since $0<L<\sqrt{\sigma\kappa}$, we have
	\begin{align}
		2L\sum_{K\in\mathcal{T}_h}[\nabla\Pi^K_1 p^n_h,\bm E^n_h]_{e,K} &\leq \kappa\sum_{K\in\mathcal{T}_h}[\nabla\Pi^K_1 p^n_h,\nabla\Pi^K_1 p^n_h]_{e,K} + \sigma[\bm E^n_h,\bm E^n_h]_e \notag\\
		&\leq \kappa[p^n_h,p^n_h]_r + \sigma[\bm E^n_h,\bm E^n_h]_e.\label{stability-2}
	\end{align}
	Meanwhile, we have the following inequality:
	\begin{align}
		\tau[\T\bm E^n_h,\bm E^n_h]_e &= \frac{\tau^2}{2}\||\T\bm E^n_h\||^2_e + \frac{1}{2}\||\bm E^n_h\||^2_e - \frac{1}{2}\||\bm E^{n-1}_h\||^2_e \notag\\
		&\geq\frac{1}{2}\||\bm E^n_h\||^2_e - \frac{1}{2}\||\bm E^{n-1}_h\||^2_e.\label{stability-3}
	\end{align}
	Similarly,
	\begin{align*}
		&\tau[\T\bm H^n_h,\bm H^n_h]_f \geq \frac{1}{2}\||\bm H^n_h\||^2_f - \frac{1}{2}\||\bm H^{n-1}_h\||^2_f,\\
		&\tau[\T p^n_h,p^n_h]_m\geq \frac{1}{2}\||p^n_h\||^2_m - \frac{1}{2}\||p^{n-1}_h\||^2_m,\\
		&\tau(\T\div\bm u^n_h,\div\bm u^n_h)\geq\frac{1}{2}\|\div\bm u^n_h\|^2_0 - \frac{1}{2}\|\div\bm u^{n-1}_h\|^2_0,\\
		&\tau[\T\bm u^n_h,\bm u^n_h]_s\geq\frac{1}{2}\||\bm u^n_h\||^2_s-\frac{1}{2}\||\bm u^{n-1}_h\||^2_s.
	\end{align*}
	Combining the above inequalities, \eqref{stability-1}--\eqref{stability-3} and the Cauchy--Schwarz inequality, we get
	\begin{equation*}
		\begin{split}
			&\varepsilon\||\bm E^n_h\||^2_e + \mu\||\bm H^n_h\||^2_f + c_0\||p^n_h\||^2_m + (\lambda+G)\|\div\bm u^n_h\|^2_0 + G\||\bm u^n_h\||^2_s	\\
			&\leq\varepsilon\||\bm E^{n-1}_h\||^2_e + \mu\||\bm H^{n-1}_h\||^2_f + c_0\||p^{n-1}_h\||^2_m + (\lambda+G)\|\div\bm u^{n-1}_h\|^2_0 + G\||\bm u^{n-1}_h\||^2_s \\
			&\quad + \tau\|\bm j_n\|^2_0 + \tau\|g_n\|^2_0 + \tau\||\bm E^n_h\||^2_e + \tau\||p^n_h\||^2_m.
		\end{split}
	\end{equation*}
	Furthermore, we obtain
	\begin{equation*}
		\begin{split}
			&\varepsilon\||\bm E^n_h\||^2_e + \mu\||\bm H^n_h\||^2_f + c_0\||p^n_h\||^2_m + (\lambda+G)\|\div\bm u^n_h\|^2_0 + G\||\bm u^n_h\||^2_s	\\
			&\leq\varepsilon\||\bm E^0_h\||^2_e + \mu\||\bm H^0_h\||^2_f + c_0\||p^0_h\||^2_m + (\lambda+G)\|\div\bm u^0_h\|^2_0 + G\||\bm u^0_h\||^2_s \\
			&\quad + \tau\sum_{i=1}^{n}\big(\|\bm j_i\|^2_0 + \|g_i\|^2_0 + \||\bm E^i_h\||^2_e + \||p^i_h\||^2_m\big).
		\end{split}
	\end{equation*}
	According to Lemma \ref{lemma-interpolation},
	\begin{equation*}
		\begin{split}
			&\||p^0_h\||_m\lesssim \|\mathrm{I}_rp_0\|_0\lesssim \|p_0\|_2,\\
			&(\lambda+G)\|\div\bm u^0_h\|^2_0+G\||\bm u^0_h\||^2_s\lesssim (\lambda+G)\|\div\mathrm{I}_s\bm u_0\|^2_0+G|\mathrm{I}_s\bm u_0|^2_1	\\
			&\qquad\qquad\qquad\qquad\qquad\lesssim (\lambda+G)\|\div\bm u_0\|^2_0 + G\|\bm u_0\|^2_2,	\\
			&\||\bm H^0_h\||_f\lesssim \|\mathrm{I}_f\bm H_0\|_0\lesssim \|\bm H_0\|_1,	\\
			&\||\bm E^0_h\||_e\lesssim \|\mathrm{I}_e\bm E_0\|_0\lesssim \|\bm E_0\|_2.
		\end{split}
	\end{equation*}
	Then, by the discrete Gronwall inequality, Lemmas \ref{lemma-regularity-1} and \ref{lemma-regularity-2}, we obtain
	\begin{equation*}
		\varepsilon\||\bm E^n_h\||^2_e + \mu\||\bm H^n_h\||^2_f + c_0\||p^n_h\||^2_m + (\lambda+G)\|\div\bm u^n_h\|^2_0 + G\||\bm u^n_h\||^2_s\lesssim 1.
	\end{equation*}
\end{proof}

\section{Error estimates}

For any $\bm v, \bm w\in H^1(\Omega,\mathbb{R}^3)$, $\bm v_h, \bm w_h\in\bm V^s_h$ and $K\in\mathcal{T}_h$, we set
\begin{equation*}
	\begin{split}
		&a^K(\bm v,\bm w) = (\lambda+G)(\div\bm v,\div\bm w)_K + G(\nabla\bm v,\nabla\bm w)_K,\\
		&a^K_h(\bm v_h,\bm w_h) = (\lambda+G)(\div\bm v_h,\div\bm w_h)_K + G[\bm v_h,\bm w_h]_{s,K},
	\end{split}
\end{equation*}
and then $a(\bm v,\bm w): = \sum_{K\in\mathcal{T}_h}a^K(\bm v,\bm w)$ and $a_h(\bm v_h,\bm w_h): = \sum_{K\in\mathcal{T}_h}a^K_h(\bm v_h,\bm w_h)$. For any $p\in H^1_0(\Omega)$, we define the projection operator $R_h:H^1_0(\Omega)\rightarrow V^r_h$ by
\begin{equation}
	[R_hp,q_h]_r = (\nabla p,\nabla q_h)\quad\forall\,q_h\in V^r_h. \label{Ritz-Rh}
\end{equation}
Further, for any $\bm u\in H^1_0(\Omega,\mathbb{R}^3)$, we define the projection operator $Q_h:H^1_0(\Omega,\mathbb{R}^3)\rightarrow \bm V^s_h$ as follows:
\begin{equation}
	a_h(Q_h\bm u,\bm v_h) - (R_hp,\alpha\div\bm v_h) = a(\bm u,\bm v_h) - (p,\alpha\div\bm v_h)\quad\forall\,\bm v_h\in\bm V^s_h, \label{Ritz-Qh}
\end{equation}
where $p$ is the pressure component associated with $\bm u$. For the projection operators $R_h$ and $Q_h$, we have the following error estimates.

\begin{lemma}\label{lemma-Ritz-Rh}
	For any $p\in H^2(\Omega)\cap H^1_0(\Omega)$ and $K\in\mathcal{T}_h$, it holds
	\begin{align}
		&\|p-R_hp\|_0 + h|p-R_hp|_1 \lesssim h^2|p|_2,\label{estimate-p-Rhp}	\\
		&\sum_{K\in\mathcal{T}_h}\|p-\Pi^K_1R_hp\|_{0,K} + \sum_{K\in\mathcal{T}_h}h_K|p-\Pi^K_1R_hp|_{1,K} \lesssim h^2|p|_2.\label{estimate-p-PiRhp}
	\end{align}
\end{lemma}

\begin{proof}
	Let $p_{\pi}|_K = \Pi^{\nabla,K}_1p$ for all $K\in\mathcal{T}_h$. For any $q_h\in V^r_h$, using \eqref{Ritz-Rh}, we have
	\begin{equation*}
		\begin{split}
			[R_hp-\mathrm{I}_rp,q_h]_r &= \sum_{K\in\mathcal{T}_h}\big((\nabla(p-p_{\pi}),\nabla q_h)_{K} + [p_{\pi}-\mathrm{I}_rp,q_h]_{r,K}\big)	\\
			&\lesssim \sum_{K\in\mathcal{T}_h}(|p-p_{\pi}|_{1,K}|q_h|_{1,K} + |p_{\pi}-\mathrm{I}_rp|_{1,K}|q_h|_{1,K})	\\
			&\lesssim h|p|_2|q_h|_1.
		\end{split}
	\end{equation*}
	Let $q_h = R_hp-\mathrm{I}_rp$. Using the coercivity of $[\cdot,\cdot]_r$, we get
	\begin{equation*}
		|R_hp-\mathrm{I}_rp|_1\lesssim h|p|_2.
	\end{equation*}
	The triangle inequality and Lemma \ref{lemma-interpolation} show that
	\begin{equation*}
		|p-R_hp|_1\lesssim h|p|_2.
	\end{equation*}
	We consider the auxiliary problem
	\begin{equation*}
		-\Delta \phi = p-R_hp,
	\end{equation*}
	where $\phi\in H^1_0(\Omega)\cap H^2(\Omega)$ and $\|\phi\|_2\lesssim \|p-R_hp\|_0$. Then, we obtain
	{\allowdisplaybreaks
	\begin{align*}
		&\|p-R_hp\|^2_0 = (p-R_hp,-\Delta\phi) \\
		=& (\nabla(p-R_hp),\nabla(\phi-\mathrm{I}_r\phi)) + (\nabla(p-R_hp),\nabla\mathrm{I}_r\phi)	\\
		=& (\nabla(p-R_hp),\nabla(\phi-\mathrm{I}_r\phi)) \\
		&+ \sum_{K\in\mathcal{T}_h}\big([R_hp-p_{\pi},\mathrm{I}_r\phi-\Pi^{\nabla,K}_1\phi]_{r,K} + (\nabla(p_{\pi}-R_hp),\nabla(\mathrm{I}_r\phi-\Pi^{\nabla,K}_1\phi))_K\big)	\\
		\lesssim& h^2|p|_2|\phi|_2\lesssim h^2|p|_2\|p-R_hp\|_0,
	\end{align*}}
	which readily implies
	\begin{equation*}
		\|p-R_hp\|_0\lesssim h^2|p|_2.
	\end{equation*}
	
	Using the triangle inequality, the Poincar\'e inequality and \eqref{estimate-p-Rhp} gives
	\begin{equation*}
		\begin{split}
			&\sum_{K\in\mathcal{T}_h}\|p-\Pi^K_1R_hp\|_{0,K}\leq \|p-R_hp\|_0+\sum_{K\in\mathcal{T}_h}\|R_hp-\Pi^K_1R_hp\|_{0,K} \\
			&\lesssim \|p-R_hp\|_0+\sum_{K\in\mathcal{T}_h}h_K|R_hp-\Pi^K_1R_hp|_{1,K}	\\
			&\lesssim \|p-R_hp\|_0+\sum_{K\in\mathcal{T}_h}h_K(|R_hp-p|_{1,K}+|p-p_{\pi}|_{1,K}+|\Pi^K_1(p_{\pi}-R_hp)|_{1,K}) \\
			&\lesssim h^2|p|_2.
		\end{split}
	\end{equation*}
	Similarly,
	\begin{equation*}
		\begin{split}
			&\sum_{K\in\mathcal{T}_h}|p-\Pi^K_1R_hp|_{1,K}\leq |p-R_hp|_1+\sum_{K\in\mathcal{T}_h}|R_hp-\Pi^K_1R_hp|_{1,K}	\\
			&\leq |p-R_hp|_1+\sum_{K\in\mathcal{T}_h}(|R_hp-p|_{1,K}+|p-p_{\pi}|_{1,K}+|\Pi^K_1(p_{\pi}-R_hp)|_{1,K}) \\
			&\lesssim h|p|_2.
		\end{split}
	\end{equation*}
\end{proof}

\begin{lemma}\label{lemma-Ritz-Qh}
	Let $p\in H^2(\Omega)\cap H^1_0(\Omega)$. For any $\bm u\in H^2(\Omega,\mathbb{R}^3)\cap H^1_0(\Omega,\mathbb{R}^3)$, it holds
	\begin{equation*}
		\begin{split}
			&|\bm u-Q_h\bm u|_1\lesssim h\|\bm u\|_2 + h(\lambda+G)\|\div\bm u\|_1 + h^2\|p\|_2,	\\
			&(\lambda+G)\|\div(Q_h\bm u-\mathrm{I}_s\bm u)\|^2_0\lesssim h^2\|\bm u\|^2_2 + h^2(\lambda+G)^2\|\div\bm u\|^2_1 + h^4\|p\|^2_2.
		\end{split}
	\end{equation*}
\end{lemma}

\begin{proof}
	Let $\bm u_{\pi}|_K = \bm \Pi^{\nabla,K}_1\bm u$ for all $K\in\mathcal{T}_h$. According to \eqref{Ritz-Qh}, \eqref{estimate-p-PiRhp} and Lemma \ref{lemma-interpolation},
	\begin{align}
		a_h&(Q_h\bm u-\mathrm{I}_s\bm u,Q_h\bm u-\mathrm{I}_s\bm u) = a(\bm u,Q_h\bm u-\mathrm{I}_s\bm u) - a_h(\mathrm{I}_s\bm u,Q_h\bm u-\mathrm{I}_s\bm u) \notag\\
		&\quad + (R_hp-p,\alpha\div(Q_h\bm u-\mathrm{I}_s\bm u)) \notag\\
		& = \sum_{K\in\mathcal{T}_h}\Big(G(\nabla(\bm u-\bm u_{\pi}),\nabla(Q_h\bm u-\mathrm{I}_s\bm u))_K + G[\bm u_{\pi}-\mathrm{I}_s\bm u,Q_h\bm u-\mathrm{I}_s\bm u]_{s,K}\Big)	\notag\\
		& \quad+ (\lambda+G)(\div(\bm u-\mathrm{I}_s\bm u),\div(Q_h\bm u-\mathrm{I}_s\bm u)) + (R_hp-p,\alpha\div(Q_h\bm u-\mathrm{I}_s\bm u)) \notag\\
		&\lesssim (hG\|\bm u\|_2 + h(\lambda+G)\|\div\bm u\|_1 + h^2\|p\|_2)|Q_h\bm u-\mathrm{I}_s\bm u|_1.\label{estimate-Qhu-Isu}
	\end{align}
	It follows from \eqref{estimate-Qhu-Isu} that
	\begin{equation}\label{estimate-Qhu-Isu-h}
		|Q_h\bm u-\mathrm{I}_s\bm u|_1\lesssim hG\|\bm u\|_2 + h(\lambda+G)\|\div\bm u\|_1 + h^2\|p\|_2.
	\end{equation}
	Applying \eqref{estimate-Qhu-Isu-h}, the triangle inequality and Lemma \ref{lemma-interpolation}, we get
	\begin{equation*}
		|\bm u-Q_h\bm u|_1\lesssim h\|\bm u\|_2 + h(\lambda+G)\|\div\bm u\|_1 + h^2\|p\|_2.
	\end{equation*}
	Then, it follows from \eqref{estimate-Qhu-Isu} and \eqref{estimate-Qhu-Isu-h} that
	\begin{equation*}
		(\lambda+G)\|\div(Q_h\bm u-\mathrm{I}_s\bm u)\|^2_0\lesssim h^2\|\bm u\|^2_2 + h^2(\lambda+G)^2\|\div\bm u\|^2_1 + h^4\|p\|^2_2.
	\end{equation*}
	
\end{proof}

We are now ready to present the error estimate for the VEM.
\begin{theorem}
	Assume \textbf{A1}--\textbf{A3}. Let $(\bm E,\bm H,\bm u, p)$ and $(\bm E^n_h,\bm H^n_h,\bm u^n_h, p^n_h)$ be the solutions of the problems \eqref{variational-formulation} and \eqref{discrete-problem}, respectively. For all $0\leq n\leq N$, we have
	\begin{equation*}
		\|\bm E(t_n)-\bm E^n_h\|_0 + \|\bm H(t_n)-\bm H^n_h\|_0 + |\bm u(t_n)-\bm u^n_h|_1 + \|p(t_n)-p^n_h\|_0\lesssim \tau+h.
	\end{equation*}
\end{theorem}

\begin{proof}
	We define $\bm E^n_{\pi}|_K = \bm\Pi^K_0\bm E_n$, $\bm H^n_{\pi}|_K = \bm\Pi^K_0\bm H_n$, $\bm u^n_{\pi}|_K = \bm\Pi^{\nabla,K}_1\bm u_n$, $p^n_{\pi}|_K = \Pi^{\nabla,K}_1p_n$. We split each error into two parts as follows:
	\begin{equation*}
		\begin{split}
			&\bm E_n-\bm E^n_h = \bm E_n-\bm P_e\bm E_n + \bm P_e\bm E_n -\bm E^n_h =: \bm\xi_n + \bm\xi^n_h,	\\
			&\bm H_n-\bm H^n_h = \bm H_n-\bm P_f\bm H_n + \bm P_f\bm H_n -\bm H^n_h =: \bm\phi_n + \bm\phi^n_h,	\\
			&\bm u_n-\bm u^n_h = \bm u_n-Q_h\bm u_n + Q_h\bm u_n -\bm u^n_h =: \bm\delta_n + \bm\delta^n_h,	\\
			&p_n-p^n_h = p_n-R_hp_n+R_hp_n-p^n_h =: \eta_n + \eta^n_h.
		\end{split}
	\end{equation*}
	Subtracting \eqref{discrete-problem} from \eqref{variational-formulation} and testing with $\bm D_h\in\bm V^e_h$, we obtain after rearrangement that
	{\allowdisplaybreaks
	\begin{align}
		&[\varepsilon\T\bm\xi^n_h,\bm D_h]_e = [\varepsilon\T(\bm P_e\bm E_n-\bm E^n_{\pi}),\bm D_h]_e + (\varepsilon\T(\bm E^n_{\pi}-\bm E_n),\bm D_h) + (\varepsilon(\T\bm E_n-\partial_t\bm E_n),\bm D_h) \notag\\
		& \quad + (\varepsilon\partial_t\bm E_n,\bm D_h) - [\varepsilon\T\bm E^n_h,\bm D_h]_e \notag\\
		& = [\varepsilon\T(\bm P_e\bm E_n-\bm E^n_{\pi}),\bm D_h]_e + (\varepsilon\T(\bm E^n_{\pi}-\bm E_n),\bm D_h) + (\varepsilon(\T\bm E_n-\partial_t\bm E_n),\bm D_h) \notag\\
		& \quad + [\sigma\bm E^n_h,\bm D_h]_e-(\sigma\bm E_n,\bm D_h) + (\bm H_n,\curl\bm D_h)-[\bm H^n_h,\curl\bm D_h]_f \notag\\
		& \quad + (L\nabla p_n,\bm D_h)-\sum_{K\in\mathcal{T}_h}[L\nabla\Pi^K_1 p^n_h,\bm D_h]_{e,K} + \sum_{K\in\mathcal{T}_h}(\bm j_n-\bm\Pi^K_0\bm j_n,\bm D_h)_K \notag\\
		& =:\sum_{i=1}^{8}\mathrm{I}^E_i(\bm D_h) + [\bm\phi^n_h,\curl\bm D_h]_f.\label{error-E}
	\end{align}
	}
	where
	{\allowdisplaybreaks
	\begin{align*}
		\mathrm{I}^E_1(\bm D_h) &= [\varepsilon\T(\bm P_e\bm E_n-\bm E^n_{\pi}),\bm D_h]_e,\quad\mathrm{I}^E_2(\bm D_h) = (\varepsilon\T(\bm E^n_{\pi}-\bm E_n),\bm D_h),	\\
		\mathrm{I}^E_3(\bm D_h) &= (\varepsilon(\T\bm E_n-\partial_t\bm E_n),\bm D_h),\quad \mathrm{I}^E_4(\bm D_h) = -\sigma[\bm\xi^n_h,\bm D_h]_e,	\\
		\mathrm{I}^E_5(\bm D_h) &= \sigma\Big([\bm P_e\bm E_n-\bm E^n_{\pi},\bm D_h]_e + (\bm E^n_{\pi}-\bm E_n,\bm D_h)\Big),\quad\mathrm{I}^E_6(\bm D_h) = \sum_{K\in\mathcal{T}_h}(\bm j_n-\bm\Pi^K_0\bm j_n,\bm D_h)_K,	\\
		\mathrm{I}^E_7(\bm D_h) &= L\sum_{K\in\mathcal{T}_h}\Big((\nabla p_n-\nabla p^n_{\pi},\bm D_h)_K + [\nabla p^n_{\pi}-\nabla\Pi^K_1 R_hp_n,\bm D_h]_{e,K}\Big),	\\
		\mathrm{I}^E_8(\bm D_h) &= L\sum_{K\in\mathcal{T}_h}[\nabla\Pi^K_1\eta^n_h,\bm D_h]_{e,K}.
	\end{align*}
	}
	For the magnetic field, testing the second equation of \eqref{discrete-problem} with $\bm B_h\in\bm V^f_h$ yields
	{\allowdisplaybreaks
	\begin{align}
		&[\mu\T\bm\phi^n_h,\bm B_h]_f = [\mu\T(\bm P_f\bm H_n-\bm H^n_{\pi}),\bm B_h]_f + (\mu\T(\bm H^n_{\pi}-\bm H_n),\bm B_h) + (\mu(\T\bm H_n-\partial_t\bm H_n),\bm B_h) \notag\\
		& \quad + (\mu\partial_t\bm H_n,\bm B_h) - [\mu\T\bm H^n_h,\bm B_h]_f \notag\\
		& = \sum_{i=1}^{4}\mathrm{I}^H_i(\bm B_h) - [\curl\bm\xi^n_h,\bm B_h]_f,\label{error-H}
	\end{align}
	}
	where
	{\allowdisplaybreaks
	\begin{align*}
		\mathrm{I}^H_1(\bm B_h) &= [\mu\T(\bm P_f\bm H_n-\bm H^n_{\pi}),\bm B_h]_f,\quad\mathrm{I}^H_2(\bm B_h) = (\mu\T(\bm H^n_{\pi}-\bm H_n),\bm B_h),	\\
		\mathrm{I}^H_3(\bm B_h) &= (\mu(\T\bm H_n-\partial_t\bm H_n),\bm B_h),	\\
		\mathrm{I}^H_4(\bm B_h) &= \sum_{K\in\mathcal{T}_h}\Big([\curl(\bm P_e\bm E_n-\bm E^n_{\pi,1}),\bm B_h]_{f,K} + (\curl(\bm E^n_{\pi,1}-\bm E_n),\bm B_h)_K\Big),
	\end{align*}
	}
	and $\bm E^n_{\pi,1}|_K\in\mathbb{P}_1(K,\mathbb{R}^3)$ is the polynomial function that satisfies
	\begin{equation*}
		\|\bm E_n - \bm E^n_{\pi,1}\|_{0,K} + h_K|\bm E_n - \bm E^n_{\pi,1}|_{1,K} \lesssim h^2_K|\bm E_n|_{2,K}.
	\end{equation*}
	For the displacement, testing the elasticity equation with $\bm v_h\in\bm V^s_h$ gives
	\begin{align}
		a_h(\bm\delta^n_h,\bm v_h) - ({\eta}^n_h,\alpha\div\bm v_h) & = a_h(Q_h\bm u_n,\bm v_h) - (R_hp_n,\alpha\div\bm v_h) - a_h(\bm u^n_h,\bm v_h) + (p^n_h,\alpha\div\bm v_h) \notag\\
		& = a(\bm u_n,\bm v_h) - (p_n,\alpha\div\bm v_h) - a_h(\bm u^n_h,\bm v_h) + (p^n_h,\alpha\div\bm v_h) = 0.\label{error-u}
	\end{align}
	Using the identity $\nabla\Pi^K_1 q_h = \Pi^K_0\nabla q_h$ for all $q_h\in V^r_h$ and $K\in\mathcal{T}_h$, and testing the pressure equation with $q_h\in V^r_h$, we obtain
	{\allowdisplaybreaks
	\begin{align}
		&[c_0\T\eta^n_h,q_h]_m + (q_h,\alpha\T\div\bm\delta^n_h) + \kappa[\eta^n_h,q_h]_r = c_0[\T(R_hp_n-p^n_{\pi}),q_h]_m + c_0(\T(p^n_{\pi}-p_n),q_h) \notag\\
		& \quad + c_0(\T p_n-\partial_tp_n,q_h) + c_0(\partial_tp_n,q_h) - c_0[\T p^n_h,q_h]_m + \alpha(q_h,\T\div(Q_h\bm u_n-\bm u_n)) \notag\\
		& \quad + \alpha(q_h,\T\div\bm u_n-\partial_t\div\bm u_n) + \alpha(q_h,\partial_t\div\bm u_n) - \alpha(q_h,\T\div\bm u^n_h)	\notag\\
		& \quad + \kappa(\nabla p_n,\nabla q_h)-\kappa[p^n_h,q_h]_r \notag\\
		& = \sum_{i=1}^{9}\mathrm{I}^p_i(q_h),\label{error-p}
	\end{align}
	}
	where
	{\allowdisplaybreaks
	\begin{align*}
		\mathrm{I}^p_1(q_h) &= c_0[\T(R_hp_n-p^n_{\pi}),q_h]_m,\quad \mathrm{I}^p_2(q_h) = c_0(\T(p^n_{\pi}-p_n),q_h),	\\
		\mathrm{I}^p_3(q_h) &= c_0(\T p_n-\partial_tp_n,q_h),\quad \mathrm{I}^p_4(q_h) = \alpha(q_h,\T\div(Q_h\bm u_n-\bm u_n)),	\\
		\mathrm{I}^p_5(q_h) &= \alpha(q_h,\T\div\bm u_n-\partial_t\div\bm u_n),\quad \mathrm{I}^p_6(q_h) = L\sum_{K\in\mathcal{T}_h}(\bm E_n - \bm E^n_{\pi},\nabla q_h)_K,	\\
		\mathrm{I}^p_7(q_h) &= L\sum_{K\in\mathcal{T}_h}[\bm E^n_{\pi} - \bm P_e\bm E_n,\Pi^K_0\nabla q_h]_{e,K},\quad \mathrm{I}^p_8(q_h) = L\sum_{K\in\mathcal{T}_h}[\bm\xi^n_h,\Pi^K_0\nabla q_h]_{e,K},	\\
		\mathrm{I}^p_9(q_h) &=\sum_{K\in\mathcal{T}_h}(g_n-\Pi^K_1g_n,q_h)_K.
	\end{align*}
	}

	Let $\iota>0$ be a fixed constant. Bounding each term in \eqref{error-E} by the Cauchy--Schwarz inequality and the triangle inequality gives
	{\allowdisplaybreaks
	\begin{align*}
		\tau \mathrm{I}^E_1(\bm D_h) & \lesssim \tau\|\T(\bm P_e\bm E_n-\bm E^n_{\pi})\|^2_0+\tau\||\bm D_h\||^2_e \\
		& \lesssim \tau(\|\T(\bm P_e\bm E_n-\bm E_n)\|^2_0 + \|\T(\bm E_n-\bm E^n_{\pi})\|^2_0)+\tau\||\bm D_h\||^2_e	\\
		& \lesssim \tau h^2\|\T\bm E_n\|^2_2 + \tau h^2\|\T\bm E_n\|^2_1 + \tau\||\bm D_h\||^2_e \\
		& \lesssim h^2\int_{t_{n-1}}^{t_n}\|\partial_t\bm E(\theta)\|^2_2\,\mathrm{d}\theta + \tau\||\bm D_h\||^2_e,	\\
		\tau \mathrm{I}^E_2(\bm D_h) & \lesssim \tau\|\T(\bm E_n-\bm E^n_{\pi})\|^2_0 + \tau\||\bm D_h\||^2_e \lesssim h^2\int_{t_{n-1}}^{t_n}\|\partial_t\bm E(\theta)\|^2_2\,\mathrm{d}\theta + \tau\||\bm D_h\||^2_e,	\\
		\tau \mathrm{I}^E_3(\bm D_h) & \lesssim \tau\|\T\bm E_n-\partial_t\bm E_n\|^2_0 + \tau\||\bm D_h\||^2_e \lesssim \tau^2\int_{t_{n-1}}^{t_n}\|\partial_{tt}\bm E(\theta)\|^2_0\,\mathrm{d}\theta + \tau\||\bm D_h\||^2_e,	\\
		\tau \mathrm{I}^E_4(\bm D_h) & \lesssim \tau \||\bm\xi^n_h\||_e\||\bm D_h\||_e,	\\
		\tau \mathrm{I}^E_5(\bm D_h) & \lesssim \tau(\|\bm E_n-\bm P_e\bm E_n\|^2_0 + \|\bm E_n-\bm E^n_{\pi}\|^2_0) + \tau\||\bm D_h\||^2_e \lesssim \tau h^2\|\bm E_n\|^2_2 + \tau\||\bm D_h\||^2_e,	\\
		\tau \mathrm{I}^E_6(\bm D_h) & \lesssim \tau\sum_{K\in\mathcal{T}_h}\|\bm j_n-\bm\Pi^K_0\bm j_n\|^2_{0,K} + \tau\||\bm D_h\||^2_e \lesssim \tau h^2\|\bm j_n\|^2_1 + \tau\||\bm D_h\||^2_e,	\\
		\tau \mathrm{I}^E_7(\bm D_h) & \lesssim \tau\sum_{K\in\mathcal{T}_h}\Big(|p_n-p^n_{\pi}|^2_{1,K}+|p_n-\Pi^K_1R_hp_n|^2_{1,K}\Big) + \tau\||\bm D_h\||^2_e \lesssim \tau h^2\|p_n\|^2_2 + \tau\||\bm D_h\||^2_e,	\\
		\tau \mathrm{I}^E_8(\bm D_h) & = L\tau\sum_{K\in\mathcal{T}_h}[\Pi^K_0\nabla\eta^n_h,\bm D_h]_{e,K} \leq C(\iota)\tau\||\bm D_h\||^2_e + \iota\tau\kappa|\eta^n_h|^2_1.
	\end{align*}
	}
	The analogous bounds for the magnetic-field residuals \eqref{error-H} read
	{\allowdisplaybreaks
	\begin{align*}
			\tau \mathrm{I}^H_1(\bm B_h) & \lesssim \tau (\|\T(\bm H_n-\bm P_f\bm H_n)\|^2_0+\|\T(\bm H_n-\bm H^n_{\pi})\|^2_0) + \tau\||\bm B_h\||^2_f \\
			& \lesssim \tau h^2\|\T\bm H_n\|^2_1 + \tau\||\bm B_h\||^2_f \lesssim h^2\int_{t_{n-1}}^{t_n}\|\partial_t\bm H(\theta)\|^2_1\,\mathrm{d}\theta + \tau\||\bm B_h\||^2_f,	\\
			\tau \mathrm{I}^H_2(\bm B_h) & \lesssim \tau\|\T(\bm H_n-\bm H^n_{\pi})\|^2_0 + \tau\||\bm B_h\||^2_f\lesssim h^2\int_{t_{n-1}}^{t_n}\|\partial_t\bm H(\theta)\|^2_1\,\mathrm{d}\theta + \tau\||\bm B_h\||^2_f,	\\
			\tau \mathrm{I}^H_3(\bm B_h) & \lesssim \tau\|\T\bm H_n-\partial_t\bm H_n\|^2_0 + \tau\||\bm B_h\||^2_f\lesssim \tau^2\int_{t_{n-1}}^{t_n}\|\partial_{tt}\bm H(\theta)\|^2_0\,\mathrm{d}\theta + \tau\||\bm B_h\||^2_f,	\\
			\tau \mathrm{I}^H_4(\bm B_h) & \lesssim \tau(\|\curl(\bm E_n-\bm P_e\bm E_n)\|^2_0 + \sum_{K\in\mathcal{T}_h}|\bm E_n-\bm E^n_{\pi,1}|^2_{1,K}) + \tau\||\bm B_h\||^2_f \\
			& \lesssim \tau h^2\|\bm E_n\|^2_2 + \tau\||\bm B_h\||^2_f.
	\end{align*}
	}
	For the pressure residuals \eqref{error-p}, we obtain
	{\allowdisplaybreaks
	\begin{align*}
		\tau \mathrm{I}^p_1(q_h) & \lesssim \tau(\|\T(p_n-R_hp_n)\|^2_0+\|\T(p_n-p^n_{\pi})\|^2_0) + \tau\||q_h\||^2_m \\
		& \lesssim \tau h^4\|\T p_n\|^2_2 + \tau\||q_h\||^2_m	\\
		& \lesssim h^4\int_{t_{n-1}}^{t_n}\|\partial_tp(\theta)\|^2_2\,\mathrm{d}\theta + \tau\||q_h\||^2_m,	\\
		\tau \mathrm{I}^p_2(q_h) & \lesssim \tau\|\T(p_n-p^n_{\pi})\|^2_0 + \tau\||q_h\||^2_m \lesssim h^4\int_{t_{n-1}}^{t_n}\|\partial_tp(\theta)\|^2_2\,\mathrm{d}\theta + \tau\||q_h\||^2_m,	\\
		\tau \mathrm{I}^p_3(q_h) & \lesssim \tau\|\T p_n-\partial_tp_n\|^2_0 + \tau\||q_h\||^2_m\lesssim \tau^2\int_{t_{n-1}}^{t_n}\|\partial_{tt}p(\theta)\|^2_0\,\mathrm{d}\theta + \tau\||q_h\||^2_m,	\\
		\tau \mathrm{I}^p_4(q_h) & \lesssim \tau|\T(\bm u_n-Q_h\bm u_n)|^2_1 + \tau\||q_h\||^2_m \\
		& \lesssim \tau h^2\|\T\bm u_n\|^2_2 + \tau h^2(\lambda+G)^2\|\T\div\bm u_n\|^2_1 + \tau h^4\|\T p_n\|^2_2	\\
		& \lesssim h^2\int_{t_{n-1}}^{t_n}\|\partial_t\bm u(\theta)\|^2_2\,\mathrm{d}\theta + h^2(\lambda+G)^2\int_{t_{n-1}}^{t_n}\|\div\partial_t\bm u(\theta) \|^2_1\,\mathrm{d}\theta \\
		& \quad + h^4\int_{t_{n-1}}^{t_n}\|\partial_tp(\theta)\|^2_2\,\mathrm{d}\theta + \tau\||q_h\||^2_m,	\\
		\tau \mathrm{I}^p_5(q_h) & \lesssim \tau\|\T\div\bm u_n-\partial_t\div\bm u_n\|^2_0 + \tau\||q_h\||^2_m \\
		& \lesssim \tau^2\int_{t_{n-1}}^{t_n}\|\partial_{tt}\bm u(\theta)\|^2_1\,\mathrm{d}\theta + \tau\||q_h\||^2_m,	\\
		\tau \mathrm{I}^p_6(q_h) & \leq C(\iota)\tau\|\bm E_n-\bm E^n_{\pi}\|^2_0 + \iota\tau\kappa|q_h|^2_1\leq C(\iota)\tau h^2\|\bm E_n\|^2_1 + \iota\tau\kappa|q_h|^2_1,	\\
		\tau \mathrm{I}^p_7(q_h) & \leq C(\iota)\tau(\|\bm E_n-\bm P_e\bm E_n\|^2_0 + \|\bm E_n-\bm E^n_{\pi}\|^2_0) + \iota\tau\kappa|q_h|^2_1 \\
		& \leq C(\iota)\tau h^2\|\bm E_n\|^2_2 + \iota\tau\kappa|q_h|^2_1,	\\
		\tau \mathrm{I}^p_8(q_h) & \leq C(\iota)\tau\||\bm\xi^n_h\||^2_e + \iota\tau\kappa|q_h|^2_1,	\\
		\tau \mathrm{I}^p_9(q_h) & \lesssim \tau\sum_{K\in\mathcal{T}_h}\|g_n-\Pi^K_1g_n\|^2_{0,K} + \tau\||q_h\||^2_m\lesssim \tau h^4\|g_n\|^2_2 + \tau\||q_h\||^2_m.
	\end{align*}
    }
	We now select the test functions as the discrete error components, namely $(\bm D_h,\bm B_h,\bm v_h,q_h) = (\bm\xi^n_h,\bm\phi^n_h,\T\bm\delta^n_h,\eta^n_h)$, and sum \eqref{error-E}--\eqref{error-p} to obtain
	\begin{align}
		&\tau[\varepsilon\T\bm\xi^n_h,\bm\xi^n_h]_e + \tau[\mu\T\bm\phi^n_h,\bm\phi^n_h]_f + \tau a_h(\bm\delta^n_h,\T\bm\delta^n_h) + \tau[c_0\T\eta^n_h,\eta^n_h]_m + \tau\kappa[\eta^n_h,\eta^n_h]_r \notag\\
		&= \sum_{i=1}^{8}\tau \mathrm{I}^E_i(\bm\xi^n_h) + \sum_{i=1}^{4}\tau \mathrm{I}^H_i(\bm\phi^n_h) + \sum_{i=1}^{9}\tau \mathrm{I}^p_i(\eta^n_h).\label{error-all-1}
	\end{align}
	By the coercivity of $[\cdot,\cdot]_r$ and the discrete estimate analogous to \eqref{stability-3}, we deduce
	{\allowdisplaybreaks
	\begin{align*}
		&\tau[\T\bm\xi^n_h,\bm\xi^n_h]_e \geq\frac{1}{2}\||\bm\xi^n_h\||^2_e - \frac{1}{2}\||\bm\xi^{n-1}_h\||^2_e,\label{error-all-2-1}\\
		&\tau[\T\bm\phi^n_h,\bm\phi^n_h]_f \geq \frac{1}{2}\||\bm\phi^n_h\||^2_f - \frac{1}{2}\||\bm\phi^{n-1}_h\||^2_f,\\
		&\tau(\T\div\bm\delta^n_h,\div\bm\delta^n_h)\geq\frac{1}{2}\|\div\bm\delta^n_h\|^2_0 - \frac{1}{2}\|\div\bm\delta^{n-1}_h\|^2_0,\\
		&\tau[\T\bm\delta^n_h,\bm\delta^n_h]_s\geq\frac{1}{2}\||\bm\delta^n_h\||^2_s-\frac{1}{2}\||\bm\delta^{n-1}_h\||^2_s,\\
		&\tau[\T\eta^n_h,\eta^n_h]_m\geq \frac{1}{2}\||\eta^n_h\||^2_m - \frac{1}{2}\||\eta^{n-1}_h\||^2_m,\\
		&\tau\kappa[\eta^n_h,\eta^n_h]_r\geq r_*\tau\kappa|\eta^n_h|^2_1,
	\end{align*}
	}
	where $r_*>0$ is a positive constant. Choosing $\iota=r_*/4$ and inserting the above estimates together with \eqref{error-all-1} and Lemmas \ref{lemma-regularity-1}--\ref{lemma-regularity-2}, we arrive at
	\begin{align*}
		&\varepsilon\||\bm\xi^n_h\||^2_e + \mu\||\bm\phi^n_h\||^2_f + (\lambda+G)\|\div\bm\delta^n_h\|^2_0 + G\||\bm\delta^n_h\||^2_s + c_0\||\eta^n_h\||^2_m	\\
		&\lesssim \varepsilon\||\bm\xi^0_h\||^2_e + \mu\||\bm\phi^0_h\||^2_f + (\lambda+G)\|\div\bm\delta^0_h\|^2_0 + G\||\bm\delta^0_h\||^2_s + c_0\||\eta^0_h\||^2_m	\\
		&\quad + \tau^2 + h^2 + \tau\sum_{i=1}^{n}(\||\bm\xi^i_h\||^2_e+\||\bm\phi^i_h\||^2_f+\||\bm\delta^i_h\||^2_s+\||\eta^i_h\||^2_m).
	\end{align*}
	It remains to bound the initial errors. By Lemmas \ref{lemma-interpolation} and \ref{lemma-projection}, we have
	\begin{equation*}
		\varepsilon\||\bm\xi^0_h\||^2_e\lesssim \|\bm P_e\bm E_0 - \mathrm{I}_e\bm E_0\|^2_0\lesssim \|\bm E_0-\bm P_e\bm E_0\|^2_0 + \|\bm E_0-\mathrm{I}_e\bm E_0\|^2_0\lesssim h^2\|\bm E_0\|^2_2.
	\end{equation*}
	Similarly, combining Lemmas \ref{lemma-interpolation}, \ref{lemma-projection}, \ref{lemma-Ritz-Qh}, and \ref{lemma-Ritz-Rh} yields
	\begin{equation*}
		\begin{split}
			&\mu\||\bm\phi^0_h\||^2_f\lesssim \|\bm P_f\bm H_0 - \mathrm{I}_f\bm H_0\|^2_0\lesssim h^2\|\bm H_0\|^2_1,	\\
			&(\lambda+G)\|\div\bm\delta^0_h\|^2_0 = (\lambda+G)\|\div(Q_h\bm u_0-\mathrm{I}_s\bm u_0)\|^2_0 \\
			&\qquad\qquad\qquad \lesssim h^2\|\bm u_0\|^2_2 + h^2(\lambda+G)^2\|\div\bm u_0\|^2_1 + h^4\|p_0\|^2_2,	\\
			&G\||\bm\delta^0_h\||^2_s\lesssim |Q_h\bm u_0-\mathrm{I}_s\bm u_0|^2_1 \\
			&\qquad\quad \lesssim h^2\|\bm u_0\|^2_2 + h^2(\lambda+G)^2\|\div\bm u_0\|^2_1 + h^4\|p_0\|^2_2,	\\
			&c_0\||\eta^0_h\||^2_m \lesssim \|R_hp_0-\mathrm{I}_rp_0\|^2_0\lesssim h^4\|p_0\|^2_2.
		\end{split}
	\end{equation*}
	Applying the discrete Grönwall inequality together with Lemmas \ref{lemma-regularity-1} and \ref{lemma-regularity-2} gives
	\begin{equation*}
		\begin{split}
			&\||\bm\xi^n_h\||^2_e + \||\bm\phi^n_h\||^2_f + \||\bm\delta^n_h\||^2_s + \||\eta^n_h\||^2_m \\
			&\lesssim \tau^2+h^2 + \tau\sum_{i=1}^{n}(\||\bm\xi^i_h\||^2_e+\||\bm\phi^i_h\||^2_f+\||\bm\delta^i_h\||^2_s+\||\eta^i_h\||^2_m) \\
			&\lesssim \tau^2+h^2,
		\end{split}
	\end{equation*}
	which, after absorbing the $\tau$-weighted sums, yields
	\begin{equation*}
		\|\bm\xi^n_h\|_0 + \|\bm\phi^n_h\|_0 + |\bm\delta^n_h|_1 + \|\eta^n_h\|_0 \lesssim \tau + h.
	\end{equation*}
	Finally, the triangle inequality combines the projection and discrete parts to give the asserted estimate.
\end{proof}

\section{Numerical experiments}

In this section, we present several numerical examples to verify the theoretical findings and demonstrate the locking-free performance of the proposed VEM. The computational domain is fixed as the unit cube $\Omega = (0,1)^3$, and the final time is chosen as $T = 0.1$. Let $N=32$, and take the time step $\tau=T/N$ so that the spatial discretization error dominates the total error. Three families of polyhedral meshes are considered: CVT meshes, prism meshes, and cube meshes (see Fig.~\ref{meshes3D}).

For the convergence study we construct the following exact solution:
{\allowdisplaybreaks
\begin{align*}
	\bm E(x,y,z,t) &= \sin t\,
	\begin{pmatrix}
		\sin(\pi x)\sin(\pi y)\sin(\pi z) \\
		\sin(\pi x)\sin(\pi y)\sin(\pi z) \\
		\sin(\pi x)\sin(\pi y)\sin(\pi z)
	\end{pmatrix},\\
	\bm H(x,y,z,t) &= \pi\cos t\,
	\begin{pmatrix}
		\sin(\pi x)\cos(\pi y)\sin(\pi z) - \sin(\pi x)\sin(\pi y)\cos(\pi z)\\
		\sin(\pi x)\sin(\pi y)\cos(\pi z) - \cos(\pi x)\sin(\pi y)\sin(\pi z)\\
		\cos(\pi x)\sin(\pi y)\sin(\pi z) - \sin(\pi x)\cos(\pi y)\sin(\pi z)
	\end{pmatrix},\\
	\bm u(x,y,z,t) &= 50\,x(1-x)y(1-y) z(1-z)\,e^{-t}\,
	\begin{pmatrix}
		x(1-x)\,\Big((1-2y)z(1-z) - y(1-y)(1-2z)\Big)\\
		y(1-y)\,\Big(x(1-x)(1-2z) - (1-2x)z(1-z)\Big)\\
		z(1-z)\,\Big((1-2x)y(1-y) - x(1-x)(1-2y)\Big)
	\end{pmatrix},\\
	p(x,y,z,t) &= e^{-t}\,\sin(\pi x)\sin(\pi y)\sin(\pi z).
\end{align*}
}

\begin{figure}[!htb]%[H]
	\centering
	\subfigure[CVT mesh]
	{\includegraphics[scale=0.45]{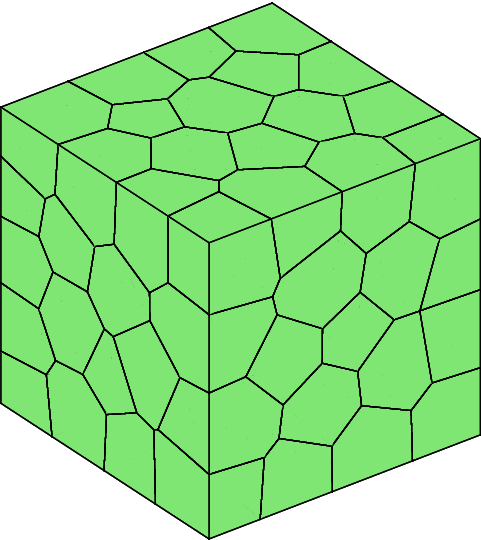}}
	\hspace{0.5cm}
	\subfigure[Prism mesh]
	{\includegraphics[scale=0.45]{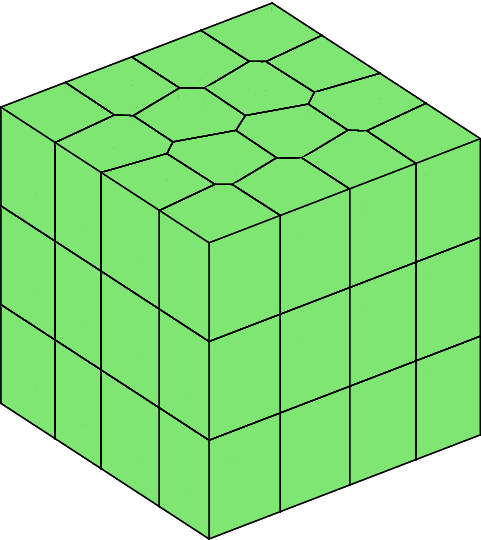}}
	\hspace{0.5cm}
	\subfigure[Cube mesh]
	{\includegraphics[scale=0.45]{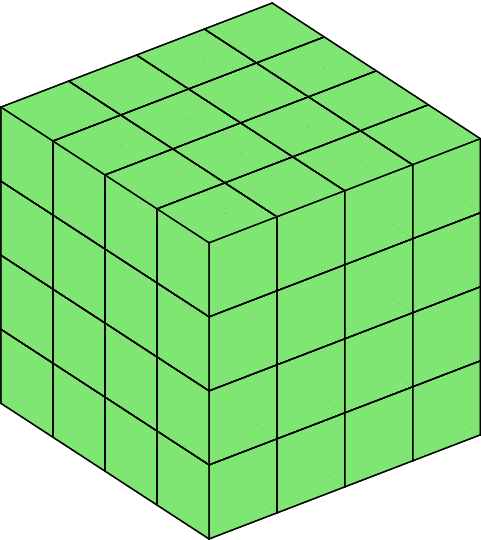}}
	\caption{Different types of meshes in three dimensions.}
	\label{meshes3D}
\end{figure}

All the parameters in the model are set as $\varepsilon = 1$, $\sigma = 1$, $\mu = 1$, $L = 1$, $G = 1$, $\alpha = 1$, $c_0 = 1$, $\kappa = 2$. The source terms $\bm j$, $\bm f$ and $g$ are computed accordingly by substituting the above expressions into \eqref{continuous-problem}.

Let $(\bm E_h^N,\bm H_h^N,\bm u_h^N,p_h^N)$ be the fully discrete VEM solution of \eqref{discrete-problem} at time $t_N = T$. The computable error indicators are defined as
{\allowdisplaybreaks
\begin{align*}
	Err_u &= \Big(\sum_{K\in\mathcal{T}_h}\bigl|\bm u(\cdot,T) - \bm\Pi^{\nabla,K}_1\bm u_h^N\bigr|_{1,K}^2\Big)^{1/2},\quad Err_p = \Big(\sum_{K\in\mathcal{T}_h}\bigl\|p(\cdot,T) - \Pi^{\nabla,K}_1 p_h^N\bigr\|_{0,K}^2\Big)^{1/2},\\
	Err_E &= \Big(\sum_{K\in\mathcal{T}_h}\bigl\|\bm E(\cdot,T) - \bm\Pi^K_0\bm E_h^N\bigr\|_{0,K}^2\Big)^{1/2},\quad Err_H = \Big(\sum_{K\in\mathcal{T}_h}\bigl\|\bm H(\cdot,T) - \bm\Pi^K_0\bm H_h^N\bigr\|_{0,K}^2\Big)^{1/2}.
\end{align*}
}

Tab.~\ref{Tab-Prism}--\ref{Tab-cube} display the errors and the estimated convergence rates under different mesh sizes $h$ for three representative values of the Lam\'e constant: $\lambda = 1$, $\lambda = 10^2$, and $\lambda = 10^5$. As $\lambda$ grows, the material becomes nearly incompressible. It is observed that all errors decrease steadily with mesh refinement, irrespective of the magnitude of $\lambda$. Moreover, the displacement and pressure errors remain uniformly bounded and exhibit no visible deterioration as $\lambda$ increases, which confirms the locking-free nature of the proposed VEM.

\begin{table}[!htb]
	\centering
	\renewcommand{\arraystretch}{1.1}
	\setlength{\tabcolsep}{5pt}
	\caption{The error results of the VEM with $\lambda = 1, 10^2, 10^5$ on the Prism meshes.}
	\begin{tabular}{lccccccccc}
		\hline
		\multicolumn{1}{c}{$\lambda$}          & $h$    & $Err_E$   & Rate & $Err_H$   & Rate & $Err_u$   & Rate & $Err_p$   & Rate \\ \hline
		\multicolumn{1}{c}{\multirow{4}{*}{$1$}} & 0.34 & 7.4565E-02 & -    & 7.7435E-01 & -    & 5.5896E-02 & -    & 6.9114E-02 & -    \\
		\multicolumn{1}{c}{}                   & 0.24 & 5.9194E-02 & 0.68 & 5.5756E-01 & 0.97 & 3.9557E-02 & 1.02 & 3.8813E-02 & 1.70 \\
		\multicolumn{1}{c}{}                   & 0.17 & 4.2001E-02 & 1.03 & 3.9174E-01 & 1.06 & 2.6810E-02 & 1.16 & 2.0437E-02 & 1.92 \\
		\multicolumn{1}{c}{}                   & 0.12 & 2.6641E-02 & 1.18 & 2.7088E-01 & 0.95 & 1.7918E-02 & 1.04 & 9.9828E-03 & 1.85 \\ \hline
		\multirow{4}{*}{$10^2$}                   & 0.34 & 7.4586E-02 & -    & 7.7435E-01 & -    & 5.3777E-02 & -    & 6.9285E-02 & -    \\
		& 0.24 & 5.9222E-02 & 0.68 & 5.5756E-01 & 0.97 & 3.8253E-02 & 1.01 & 3.8868E-02 & 1.71 \\
		& 0.17 & 4.2021E-02 & 1.03 & 3.9174E-01 & 1.06 & 2.6191E-02 & 1.13 & 2.0468E-02 & 1.92 \\
		& 0.12 & 2.6654E-02 & 1.18 & 2.7088E-01 & 0.95 & 1.7669E-02 & 1.02 & 1.0002E-02 & 1.85 \\ \hline
		\multirow{4}{*}{$10^5$}                   & 0.34 & 7.4587E-02 & -    & 7.7435E-01 & -    & 5.3755E-02 & -    & 6.9292E-02 & -    \\
		& 0.24 & 5.9223E-02 & 0.68 & 5.5756E-01 & 0.97 & 3.8243E-02 & 1.01 & 3.8869E-02 & 1.71 \\
		& 0.17 & 4.2022E-02 & 1.03 & 3.9174E-01 & 1.06 & 2.6186E-02 & 1.13 & 2.0469E-02 & 1.92 \\
		& 0.12 & 2.6655E-02 & 1.18 & 2.7088E-01 & 0.95 & 1.7667E-02 & 1.02 & 1.0003E-02 & 1.85 \\ \hline
	\end{tabular}\label{Tab-Prism}
\end{table}

\begin{table}[!htb]
	\centering
	\renewcommand{\arraystretch}{1.1}
	\setlength{\tabcolsep}{5pt}
	\caption{The error results of the VEM with $\lambda = 1, 10^2, 10^5$ on the CVT meshes.}
	\begin{tabular}{lccccccccc}
		\hline
		\multicolumn{1}{c}{$\lambda$}          & $h$    & $Err_E$   & Rate & $Err_H$   & Rate & $Err_u$   & Rate & $Err_p$   & Rate \\ \hline
		\multicolumn{1}{c}{\multirow{4}{*}{$1$}} & 0.41 & 8.4191E-02 & -    & 9.4584E-01 & -    & 6.6896E-02 & -    & 8.8120E-02 & -    \\
		\multicolumn{1}{c}{}                   & 0.30 & 7.0994E-02 & 0.56 & 7.1874E-01 & 0.91 & 5.3342E-02 & 0.75 & 6.4048E-02 & 1.06 \\
		\multicolumn{1}{c}{}                   & 0.20 & 5.5101E-02 & 0.63 & 4.8034E-01 & 1.00 & 3.4061E-02 & 1.12 & 2.9792E-02 & 1.91 \\
		\multicolumn{1}{c}{}                   & 0.15 & 3.6844E-02 & 1.30 & 3.3999E-01 & 1.12 & 2.3135E-02 & 1.25 & 1.4845E-02 & 2.25 \\ \hline
		\multirow{4}{*}{$10^2$}                   & 0.41 & 8.4201E-02 & -    & 9.4584E-01 & -    & 6.4547E-02 & -    & 8.8416E-02 & -    \\
		& 0.30 & 7.1019E-02 & 0.56 & 7.1874E-01 & 0.91 & 5.1297E-02 & 0.76 & 6.4179E-02 & 1.06 \\
		& 0.20 & 5.5127E-02 & 0.63 & 4.8034E-01 & 1.00 & 3.3122E-02 & 1.09 & 2.9835E-02 & 1.91 \\
		& 0.15 & 3.6860E-02 & 1.30 & 3.3999E-01 & 1.12 & 2.2739E-02 & 1.22 & 1.4872E-02 & 2.25 \\ \hline
		\multirow{4}{*}{$10^5$}                   & 0.41 & 8.4202E-02 & -    & 9.4584E-01 & -    & 6.4514E-02 & -    & 8.8429E-02 & -    \\
		& 0.30 & 7.1020E-02 & 0.56 & 7.1874E-01 & 0.91 & 5.1282E-02 & 0.76 & 6.4184E-02 & 1.06 \\
		& 0.20 & 5.5128E-02 & 0.63 & 4.8034E-01 & 1.00 & 3.3117E-02 & 1.09 & 2.9836E-02 & 1.91 \\
		& 0.15 & 3.6861E-02 & 1.30 & 3.3999E-01 & 1.12 & 2.2737E-02 & 1.22 & 1.4873E-02 & 2.25 \\ \hline
	\end{tabular}\label{Tab-CVT}
\end{table}

\begin{table}[!htb]
	\centering
	\renewcommand{\arraystretch}{1.1}
	\setlength{\tabcolsep}{5pt}
	\caption{The error results of the VEM with $\lambda = 1, 10^2, 10^5$ on the cube meshes.}
	\begin{tabular}{lccccccccc}
		\hline
		\multicolumn{1}{c}{$\lambda$}          & $h$    & $Err_E$   & Rate & $Err_H$   & Rate & $Err_u$   & Rate & $Err_p$   & Rate \\ \hline
		\multicolumn{1}{c}{\multirow{4}{*}{$1$}} & 0.29 & 6.7884E-02 & -    & 7.0213E-01 & -    & 5.0650E-02 & -    & 5.8892E-02 & -    \\
		\multicolumn{1}{c}{}                   & 0.22 & 5.5418E-02 & 0.71 & 5.2925E-01 & 0.98 & 3.7367E-02 & 1.06 & 3.6194E-02 & 1.69 \\
		\multicolumn{1}{c}{}                   & 0.11 & 2.4107E-02 & 1.20 & 2.6569E-01 & 0.99 & 1.7500E-02 & 1.09 & 9.9130E-03 & 1.87 \\
		\multicolumn{1}{c}{}                   & 0.09 & 1.7101E-02 & 1.54 & 2.1261E-01 & 1.00 & 1.3802E-02 & 1.06 & 6.4177E-03 & 1.95 \\ \hline
		\multirow{4}{*}{$10^2$}                   & 0.29 & 6.7902E-02 & -    & 7.0213E-01 & -    & 4.8773E-02 & -    & 5.9012E-02 & -    \\
		& 0.22 & 5.5441E-02 & 0.70 & 5.2925E-01 & 0.98 & 3.6184E-02 & 1.04 & 3.6242E-02 & 1.69 \\
		& 0.11 & 2.4120E-02 & 1.20 & 2.6569E-01 & 0.99 & 1.7266E-02 & 1.07 & 9.9314E-03 & 1.87 \\
		& 0.09 & 1.7110E-02 & 1.54 & 2.1261E-01 & 1.00 & 1.3674E-02 & 1.05 & 6.4310E-03 & 1.95 \\ \hline
		\multirow{4}{*}{$10^5$}                   & 0.29 & 6.7903E-02 & -    & 7.0213E-01 & -    & 4.8751E-02 & -    & 5.9016E-02 & -    \\
		& 0.22 & 5.5441E-02 & 0.70 & 5.2925E-01 & 0.98 & 3.6172E-02 & 1.04 & 3.6243E-02 & 1.69 \\
		& 0.11 & 2.4120E-02 & 1.20 & 2.6569E-01 & 0.99 & 1.7264E-02 & 1.07 & 9.9318E-03 & 1.87 \\
		& 0.09 & 1.7110E-02 & 1.54 & 2.1261E-01 & 1.00 & 1.3673E-02 & 1.05 & 6.4313E-03 & 1.95 \\ \hline
	\end{tabular}\label{Tab-cube}
\end{table}

Finally, we examine numerically the divergence-free property of the discrete magnetic field, which is an important constraint of the model. Recall that the continuous magnetic field satisfies $\div\bm H(t)=0$ for all $t\in[0,T]$. This constraint is inherited by the fully discrete scheme. In fact, since $\curl\bm V^e_h = \bm V^f_{h,0}$ (see Remark \ref{remark-exact-sequence}), taking the divergence on both sides of \eqref{Hn=Hn-1-En} yields
\begin{equation*}
	\div\bm H^n_h = \div\bm H^{n-1}_h - \tau\mu^{-1}\div(\curl\bm E^n_h) = \div\bm H^{n-1}_h,
\end{equation*}
and hence $\div\bm H^n_h = \div\bm H^0_h$ for all $1\leq n\leq N$. Therefore, the discrete magnetic field is divergence-free at every time level provided that the initial datum $\bm H^0_h$ is divergence-free. In our implementation, we observe that $\bm H = \curl\bm W$ with the vector potential
\begin{equation*}
	\bm W(x,y,z,t) = \cos t\,\sin(\pi x)\sin(\pi y)\sin(\pi z)
	\begin{pmatrix}
		1\\ 1\\ 1
	\end{pmatrix}.
\end{equation*}
Thus, instead of evaluating the face moments of $\mathrm{I}_f\bm H_0$ by numerical quadrature, which yields the same function in exact arithmetic, we set $\bm H^0_h = \curl(\mathrm{I}_e\bm W_0)$ with $\bm W_0 = \bm W(\cdot,0)$, which belongs to $\bm V^f_{h,0}$ by Remark \ref{remark-exact-sequence}. In this way $\div\bm H^0_h = 0$ holds up to round-off error only; a direct quadrature of the face moments would, in contrast, introduce an additional error depending on the accuracy of the quadrature rule, which would then pollute $\div\bm H^n_h$ for all subsequent time levels.

To verify this property, we compute
\begin{equation*}
	Err_{\div H} = \Bigg(\sum_{K\in\mathcal{T}_h}\|\div\bm H^N_h\|^2_{0,K}\Bigg)^{1/2}
	= \Bigg(\sum_{K\in\mathcal{T}_h}|K|^{-1}\Big(\sum_{F\in\mathcal{F}(K)}\int_F\bm H^N_h\cdot\bm n_F\,\mathrm{d}S\Big)^2\Bigg)^{1/2}.
\end{equation*}

\begin{table}[!htb]
	\centering
	\renewcommand{\arraystretch}{1.1}
	\setlength{\tabcolsep}{5pt}
	\caption{$Err_{\div H}$ of the VEM with $\lambda = 1, 10^2, 10^5$ on different meshes.}
	\begin{tabular}{clccc}
		\hline
		Mesh  & $h$    & $\lambda = 1$ & $\lambda = 10^2$ & $\lambda = 10^5$ \\ \hline
		\multirow{4}{*}{Prism} & 0.34 & 5.2096E-15 & 6.1898E-15 & 5.2810E-15 \\
		& 0.24 & 1.0409E-14 & 9.7962E-15 & 9.6886E-15 \\
		& 0.17 & 1.7277E-14 & 1.7888E-14 & 1.7853E-14 \\
		& 0.12 & 3.0579E-14 & 3.0941E-14 & 3.1165E-14 \\ \hline
		\multirow{4}{*}{CVT}   & 0.41 & 5.7707E-15 & 5.2935E-15 & 5.3269E-15 \\
		& 0.30 & 7.7573E-15 & 8.3690E-15 & 8.4523E-15 \\
		& 0.20 & 1.7340E-14 & 1.7412E-14 & 1.7574E-14 \\
		& 0.15 & 3.5581E-14 & 3.5670E-14 & 3.5478E-14 \\ \hline
		\multirow{4}{*}{Cube}  & 0.29 & 6.2329E-15 & 6.8460E-15 & 7.1920E-15 \\
		& 0.22 & 1.0317E-14 & 9.4255E-15 & 9.3699E-15 \\
		& 0.11 & 1.8710E-14 & 1.8741E-14 & 1.8933E-14 \\
		& 0.09 & 3.2826E-14 & 3.2651E-14 & 3.3079E-14 \\ \hline
	\end{tabular}\label{Tab-divH}
\end{table}

Tab.~\ref{Tab-divH} reports the values of $Err_{\div H}$ at the final time $t_N = T$ for the three families of meshes and for $\lambda = 1$, $10^2$, $10^5$. All the values are of the order $10^{-15}$--$10^{-14}$, i.e. they stay at the level of machine precision, and they exhibit no deterioration as $\lambda$ increases from $1$ to $10^5$. The mild growth with respect to mesh refinement is merely the accumulation of round-off errors in the summation over the elements. These results confirm that the magnetic field produced by the proposed VEM is divergence-free at the discrete level, which agrees well with the theoretical prediction.

\section*{Declarations}
\begin{enumerate}
	\item {\bf Conflicts of interest/Competing interests.} The authors have no known competing financial interests or personal relationships that could have appeared to influence the work reported in this manuscript.
	
	\item {\bf Availability of data and material.} The datasets generated during and/or analysed during the current study are available from the corresponding author on reasonable request.
	
	\item {\bf Authors' contributions.} All authors contributed equally to this manuscript.
\end{enumerate}

\section*{Acknowledgement}
The work was supported by NSFC (Grant no. 12571390).

\section*{References}

\bibliography{reference}

\end{document}